\documentclass[12pt]{amsart}
\usepackage{amsmath,amsthm,amsfonts,amssymb}
\usepackage[
margin=2.5cm]{geometry}
\usepackage[usenames,dvipsnames]{color}
\usepackage{mathrsfs}
\usepackage[colorlinks,allcolors=blue]{hyperref}
\usepackage{enumitem}
\setlist[enumerate]{%
  leftmargin=*,%
  label=(\roman*)} 
\setlist[itemize]{leftmargin=*}
\usepackage[mathscr]{eucal}
\usepackage{multirow, array}

\usepackage{tikz,xcolor}
\usetikzlibrary{positioning, backgrounds}

\usepackage{
  tikz-qtree}
\usetikzlibrary{matrix,arrows,trees,shapes}
\usepackage{xcolor}

\usepackage{fancyvrb} 
\RecustomVerbatimEnvironment{Verbatim}{Verbatim}%
{xleftmargin=0pt,baselinestretch=0.85}

\usepackage{microtype}

\usepackage{bm} 
\usepackage{bbm} 

\usepackage[all]{xy} 
\usepackage[T1]{fontenc} 

\usepackage[yyyymmdd]{datetime}

\usepackage{fancyhdr}
\newtheorem{theorem}{Theorem}[section]
\newtheorem{lemma}[theorem]{Lemma}
\newtheorem{proposition}[theorem]{Proposition}
\newtheorem{corollary}[theorem]{Corollary}
\newtheorem{question}[theorem]{Question}

\theoremstyle{definition}%
\newtheorem{example}[theorem]{Example}
\newtheorem{remark}[theorem]{Remark}

\DeclareMathOperator{\hilb}{Hilb}
\DeclareMathOperator{\Kos}{Kos}

\DeclareMathOperator{\Hom}{Hom}%
\DeclareMathOperator{\Soc}{Soc}%
\DeclareMathOperator{\bideg}{bideg}%
\DeclareMathOperator{\Ann}{Ann}%

\newcommand{\fmm}{\mathfrak{m}}
\newcommand{\cB}{\mathcal{B}}%
\newcommand{\cF}{\mathcal{F}}%

\newcommand{\bx}{\bm{x}}
\newcommand{\bu}{\bm{u}}

\newcommand{\kk}{\mathbbm{k}}
\newcommand{\NN}{\mathbb{N}}
\newcommand{\ZZ}{\mathbb{Z}}
\newcommand{\bbA}{\mathbb{A}}

\newcommand{\llrr}[1]{ \langle #1 \rangle } 

\newif\ifhascomments \hascommentstrue
\ifhascomments
  \newcommand{\matt}[1]{{\color{red}[[\ensuremath{\spadesuit\spadesuit\spadesuit} #1]]}}
  \newcommand{\andrew}[1]{{\color{blue}[[\ensuremath{\clubsuit\clubsuit\clubsuit} #1]]}}
\else
  \newcommand{\andrew}[1]{}
  \newcommand{\matt}[1]{}
\fi

\begin{document}

\begin{abstract}
  We answer an open problem posed by Iarrobino in the '80s: is there a
  component of the punctual Hilbert scheme
  $\hilb^d(\mathcal{O}_{\mathbb{A}^n,p})$ with dimension less than
  $(n-1)(d-1)$?  We construct an infinite class of elementary
  components in $\textrm{Hilb}^d(\mathbb{A}^4)$ producing such
  examples. Our techniques also allow us to construct an explicit
  example of a local Artinian ring with trivial negative tangents,
  vanishing nonnegative obstruction space, and socle-dimension $2$.
\end{abstract}

\title{Small elementary components of Hilbert schemes of points}

\author{Matthew~Satriano}
\thanks{MS was partially supported by a Discovery Grant from the
  National Science and Engineering Research Council of Canada.}
\address{Matthew Satriano, Department of Pure Mathematics, University
  of Waterloo}
\email{msatrian@uwaterloo.ca}

\author{Andrew~P.~Staal}
\address{Andrew~P.~Staal, Department of Mathematics, University of
  Waterloo}
\email{andrew.staal@uwaterloo.ca}
\maketitle

\section{Introduction}


Hilbert schemes of points are moduli spaces of fundamental importance
in algebraic geometry, commutative algebra, and algebraic
combinatorics. Since their construction by Grothendieck
\cite{Grothendieck--1995}, they have seen broad-ranging applications
from the McKay correspondence
\cite{Ito-Nakajima--2000,Bridgeland-King-Reid--2001} to Haiman's proof
of the Macdonald positivity conjecture \cite{Haiman--2001}. In 1968,
Fogarty \cite{Fogarty--1968} proved the irreducibility of the Hilbert
scheme of points on a smooth surface. A few years later, Iarrobino
\cite{Iarrobino--1972, Iarrobino--1973} and Iarrobino--Emsalem
\cite{Iarrobino--Emsalem--1978} showed that, in contrast, for $n\geq3$
and $d$ sufficiently large, the Hilbert scheme of points
$\hilb^d(\bbA^n)$ is reducible. Since then, it has remained a
notoriously difficult problem to describe the
structure of the irreducible components of $\hilb^d(\bbA^n)$.

Only a handful of explicit constructions of irreducible components
exist in the literature, many of these constructions involving clever
new insights \cite{Iarrobino--1984, Shafarevich--1990,
  Iarrobino--Kanev--1999, CEVV--2009, Erman--Velasco--2010,
  Huibregtse--2017, Jelisiejew--2019, Jelisiejew--2020,
  Huibregtse--2021}.  Even less is known about \emph{elementary
components}, namely, irreducible components parametrizing subschemes
supported at a point.  The study of all irreducible components may be
reduced to that of elementary ones due to the fact that, generically,
every component is \'etale-locally the product of elementary ones.
Nearly all elementary components constructed thus far have dimensions
\emph{larger} than that of the main component of $\hilb^d(\bbA^n)$,
namely $nd$. The only elementary components in the literature with
dimensions shown to be less than $nd$ are examples with Hilbert
functions $(1, 4, 3)$ and $(1, 6, 6, 1)$ due to Iarrobino--Emsalem
\cite{Iarrobino--Emsalem--1978}, $(1, 5, 3)$ and $(1, 5, 4)$ due to
Shafarevich \cite{Shafarevich--1990}, $(1, 5, 3, 4)$, $(1, 5, 3, 4, 5,
6)$, and $(1, 5, 5, 7)$ due to Huibregtse
\cite{Huibregtse--2017,Huibregtse--2021}, $(1,4,10,16,17,8)$ due to
Jelisiejew \cite{Jelisiejew--2019}, and finally, one infinite family
also constructed by Jelisiejew \cite[Theorem 1.4]{Jelisiejew--2019}.

In the case of the punctual Hilbert scheme
$\hilb^d(\mathcal{O}_{\bbA^n,p})$ at a point $p$, there is a sharp
lower bound on the dimensions of its smoothable
components. Specifically, the smoothable locus $U$ of
$\hilb^d(\bbA^n)$ determines a smoothable locus $U_p = U \cap
\hilb^d(\mathcal{O}_{\bbA^n,p})$ in the punctual Hilbert scheme. Here,
$U_p$ can be reducible unlike the case of $\hilb^d(\bbA^n)$. Gaffney
proved \cite[Theorem 3.5]{Gaffney--1988} that all irreducible
components of $U_p$ have dimension at least $(n-1)(d-1)$. Moreover, Iarrobino 
identified an irreducible component realizing this lower bound,
consisting of the curvilinear points. It has remained an open problem
for over 30 years to determine whether Gaffney's bound extends to all
irreducible components of $\hilb^d(\mathcal{O}_{\bbA^n,p})$:


\begin{question}[{\cite[p.\ 310]{Iarrobino--1987}, cf.~\cite[p.\ 186]{Iarrobino--Emsalem--1978}}]  \label{q:very-small-dim}  Let $p\in\bbA^n$ be a point. Does there exist an irreducible component of $\hilb^d(\mathcal{O}_{\bbA^n,p})$ of dimension less than $(n-1)(d-1)$?\end{question}

The goal of this paper is to answer Question
\ref{q:very-small-dim}. By generalizing the original example presented
in \cite{Iarrobino--Emsalem--1978}, and making key use of Jelisiejew's
criterion \cite{Jelisiejew--2019}, we produce an infinite family of
elementary components in $\hilb^d(\bbA^4)$ with dimension less than
$3(d-1)$. Furthermore, the examples we produce are flexible in the
sense that our elementary components $Z\subset\hilb^d(\bbA^4)$ also
frequently yield new components $Z_i\subset\hilb^{d-i}(\bbA^n)$ for
small $i$; see Theorem \ref{thm:intro2}.
  
We work throughout over an algebraically closed field $\kk$ of
characteristic $0$. 

\begin{theorem}
  \label{thm:main-result}
  Let
  \[
  d=\frac{1}{2}ab(a+b)
  \]
  with $a,b\in\mathbb{Z}$ and $2\leq a\leq b$. 
  If $(a,b)\neq (2,2)$,
  then $\hilb^d(\mathcal{O}_{\bbA^4,p})$ contains an irreducible
  component of dimension less than $3(d-1)$.
\end{theorem}

Specifically, we prove Theorem \ref{thm:main-result} by showing:


\begin{theorem}
  \label{thm:intro}
  Every ideal in $S := \kk[x,y,z,w]$ of the form
  \[
  I := \llrr{x,y}^{n_1} + \llrr{z,w}^{n_2} + \llrr{xz-yw}, 
  \]
  for $n_1, n_2 \ge 2$, determines a smooth point $[I]$ of the Hilbert
  scheme of points $\hilb^{d}(\bbA^4)$, where
  \[
  d = d(n_1,n_2) := \frac{n_1n_2(n_1+n_2)}{2}.
  \]
  Moreover, the unique component containing $[I]$ is elementary of dimension
  \[
  \frac{1}{3}m^3 + mM^2 + m^2 + 2mM + M^2 -
  \frac{1}{3}m -1,
  \]
  where $m=\min(n_1,n_2)$ and $M=\max(n_1,n_2)$. This dimension is
  less than $4d$, for all $n_1,n_2\geq2$, and less than $3(d-1)$, for
  $(m,M)\notin\{(2,2),(2,3),(2,4)\}$.
\end{theorem}



Furthermore, for any $I$ in Theorem \ref{thm:intro} and any
$s\in\Soc(S/I)$, we prove that $I+\llrr{s}$ also defines a smooth
point of $\hilb^{d-1}(\bbA^4)$ belonging to a unique elementary
component. We show that one may even iterate this construction to
produce smooth points $I+\llrr{s_1,s_2,\dots,s_r}$ on unique
elementary components provided that a particular constraint holds
which relates socles to bidegrees. Notice that the ideals in Theorem
\ref{thm:intro} are bigraded, where the bidegree of a monomial
$x^{u_1}y^{u_2}z^{u_3}w^{u_4} \in S$ is defined here to be $(u_1+u_2,
u_3+u_4) \in \NN^2$. Then we have:
  
\begin{theorem}
  \label{thm:intro2}
  Let $I$ be as in Theorem \ref{thm:intro} with $n_1,n_2 \ge 3$.
  Let $s_1, s_2, \dotsc, s_{r} \in S$ define elements in $\Soc(S/I)$,
  and let
  \[
  J = I + \llrr{s_1,\dots,s_r}
  \]
  and $B=S/J$. If either
  \begin{enumerate}
  \item $r=1$, or
  \item $\Soc B = B_{(n_1-1, n_2-1)}$,
  \end{enumerate}
  then $[J]$ is a smooth point of $\hilb^{d-r}(\bbA^4)$, belonging to
  a unique elementary component.
\end{theorem}

\begin{remark}
  \label{rmk:more-details-intro2}
  The proof of Theorem \ref{thm:intro2} shows that $J$ has
  \emph{trivial negative tangents}, i.e.\ that $T^1(B/\kk,B)_{<0} =
  0$, as well as \emph{vanishing nonnegative obstruction space},
  i.e.~$T^2(B/\kk,B)_{\ge0} = 0$. See \S\ref{sub:tcc} for the
  definitions of the $T^i$-modules.
\end{remark}

We end the introduction with a brief description of some further
applications of our techniques.  First, consider the following
folklore question, an affirmative answer to which would distinguish
cactus and secant varieties \cite[Proposition
  7.4]{Buczynski--Jelisiejew--2017} (see also
\cite{Buczynska--Buczynski--2014},
\cite{Galazka--Mandziuk--Rupniewski--2020}).

\begin{question}
  \label{q:cactus}
  Does there exist a Gorenstein local Artinian algebra of the form
  $\kk[x,y,z,w]/I$ with trivial negative tangents?
\end{question}

Theorem \ref{thm:intro2} and Remark \ref{rmk:more-details-intro2} show that
\begin{equation}\label{eqn:socle-2-ideal-example}
I = \llrr{x,y}^{3} + \llrr{z,w}^{3} + \llrr{xz-yw,x^2z^2,x^2w^2,y^2z^2}
\end{equation}
has trivial negative tangents and vanishing nonnegative obstruction
space; while $S/I$ is not Gorenstein (socle-dimension $1$), it does
have socle-dimension $2$. It is possible that variants of the ideals
considered in Theorem \ref{thm:intro2} may yield an answer to Question
\ref{q:cactus}. See Example \ref{eg:3} for further details and Remark
\ref{rmk:variants} for similar examples.

It is also interesting to note that our techniques yield examples of
Hilbert schemes with at least two elementary
components.\footnote{J. Jelisiejew (personal communication) has also
found examples for large values of $d$ by selecting tuples of random
polynomials; we are unaware of other such examples in the literature.}
Theorem \ref{thm:intro} shows, for instance, that the ideal
$\llrr{x,y}^2 + \llrr{z,w}^4 + \llrr{xz-yw}$ defines a smooth point of
an elementary component of $\hilb^{24}(\bbA^4)$. On the other hand,
Theorem \ref{thm:intro2} shows that the ideal in
\eqref{eqn:socle-2-ideal-example} also defines a smooth point on an
elementary component of $\hilb^{24}(\bbA^4)$.  An explicit check then
shows that the tangent space dimensions at these two points are
different. See Examples \ref{eg:2} and \ref{eg:2'}.

\subsection*{Acknowledgments}
It is a pleasure to thank Tony Iarrobino, Ritvik Ramkumar, and Tim Ryan for helpful email exchanges.  
We especially wish to thank Joachim Jelisiejew for pointing us to Question \ref{q:very-small-dim} and for 
his comments on an earlier draft. We are extremely grateful both to Joachim Jelisiejew and Jenna
Rajchgot for many enlightening conversations which led to this project.

 

\section{Preliminaries}

We set some notation and highlight a useful tool for studying Hilbert
schemes of points.

\subsection{Basic Set-up}
\label{sub:BF}

Let $S := \kk[x,y,z,w]$ be the coordinate ring of affine space
$\bbA^4$ where $\kk$ is an algebraically closed field of
characteristic $0$.
For a vector $\bu = (u_1,u_2,u_3,u_4) \in \NN^4$, let $\bx^{\bu} :=
x^{u_1}y^{u_2}z^{u_3}w^{u_4}$ denote the corresponding monomial in $S$
and denote its degree in the standard grading by $|\bx^{\bu}| = |\bu|
:= u_1+u_2+u_3+u_4$; more generally, we use $|f|$ to denote the degree
of any homogeneous element $f$ in the standard grading.  This grading
can be refined to a bigrading on $S$, defined on monomials by
$\bideg(\bx^{\bu}) := (u_1+u_2,u_3+u_4) \in \NN^2$.  All of the ideals
$I$ and $J$ mentioned in Theorems \ref{thm:intro} and \ref{thm:intro2}
are bigraded, and thus, standard graded.  If $R$ is a $\ZZ$-graded
ring, $M$ is a graded $R$-module, and $j \in \ZZ$, then the
\emph{\bfseries $j$th twist} of $M$ is the graded $R$-module $M(j)$
satisfying $M(j)_i := M_{j+i}$, for all $i \in \ZZ$.

\subsection{The Truncated Cotangent Complex}
\label{sub:tcc}

Our approach to proving Theorems \ref{thm:intro} and \ref{thm:intro2}
requires computing certain
$T^i$-modules,
so we review the construction of the truncated cotangent complex.  We
follow \cite[\S 3]{Hartshorne--2010} closely, which itself follows
\cite{Lichtenbaum--Schlessinger--1967}.

To obtain a model of the truncated cotangent complex of a ring
homomorphism $A \to B$, we choose surjections $R_{B/A}
\twoheadrightarrow B$, with kernel denoted $I$, and $F_{B/A}
\twoheadrightarrow I$, where $R_{B/A}$ is a polynomial ring over $A$
and $F_{B/A}$ is a free $R_{B/A}$-module.  We then set $Q_{B/A}$ to be
the kernel of $F_{B/A} \twoheadrightarrow I$ and $\Kos_{B/A}$ to be
its submodule of Koszul relations.  We drop the subscripts when no
confusion should arise.  The \emph{\bfseries truncated cotangent
complex} of $B$ over $A$ is the complex $L_{B/A,\bullet}$ concentrated
in homological degrees $0,1,2$, with terms
\[
L_{B/A,\bullet} \colon \Omega_{R/A} \otimes_{R} B
\stackrel{\ d_1^{B/A}}{\longleftarrow} F \otimes_{R} B
\stackrel{\ d_2^{B/A}}{\longleftarrow} Q/\Kos,
\]
where $d_2^{B/A}$ is induced by the inclusion $Q \subseteq F$ and
$d_1^{B/A}$ is obtained by composing the map $L_1 = F \otimes_{R} B
\twoheadrightarrow I/I^2$ with the map induced by the derivation $R
\to \Omega_{R/A}$.  We sometimes use $d_i^L$ to denote the
differentials. One derives from this the \emph{\bfseries
$T^i$-modules}
\[
T^i(B/A, M) := H^i(\Hom_B(L_{B/A,\bullet}, M)),
\]
for any $B$-module $M$ and $0 \le i \le 2$.  (We also call these
\emph{\bfseries tangent cohomology modules} when convenient.)
The notation $T^i_{B/A}$ is often used when $M = B$, or simply
$T^i_{B}$, if $A = \kk$ is the base field.  When viewed as an element
of the derived category, the complex $L_{B/A,\bullet}$ is independent
of the choices of $R_{B/A}$ and $F_{B/A}$; see e.g.\ \cite[Remark
  3.3.1]{Hartshorne--2010}.  Hence, the tangent cohomology modules
depend only on the map $A \to B$.

\begin{remark}
  \label{rmk:graded}
  When $A$ and $B$ are both graded by an abelian group $G$ and the map
  $A \to B$ is a graded homomorphism, all choices in the construction
  of the truncated cotangent complex can be made to respect the
  grading.  If the cotangent modules $L_{B/A,i}$ are all finite over
  $B$ and $M$ is a graded $B$-module, then $T^i(B/A,M)$ is also
  graded.  Importantly, the nine-term long exact sequences described
  in \cite[Theorems 3.4--3.5]{Hartshorne--2010} also respect the
  grading.  This holds for our examples, where we only need $G = \ZZ$
  or $\ZZ^2$ and the gradings mentioned in \S\ref{sub:BF}.
\end{remark}


\subsection{A Comparison Theorem}
\label{sub:comparison}

We briefly describe a theorem of Jelisiejew.  Let $I$ be any ideal in
$S := \kk[x,y,z,w]$ defining a local Artinian quotient supported at $0
\in \bbA^4$.
Motivated by the Bia{\l}ynicki-Birula decomposition, Jelisiejew
defines a scheme $\hilb^+_{pts}(\bbA^4)$ and constructs a map
\[
\theta\colon\hilb^+_{pts}(\bbA^4) \times \bbA^4 \to \hilb^{pts}(\bbA^4)
\]
with the following properties.  First,
$\theta|_{\hilb^+_{pts}(\bbA^4)\times\{0\}}$ is a monomorphism and
maps $\kk$-points bijectively to subschemes of $\bbA^4$ supported at
$0$.  Second, on the level of $\kk$-points, if $[J]$ is supported at
$0$, then $\theta([J],v)$ is the point supported at $v$ obtained by
translating $[J]$.

\begin{theorem}[{\cite[Theorem 4.5]{Jelisiejew--2019}}]
  \label{thm:Jel}
  If $I$ is supported at the origin and has trivial negative tangents,
  then $\theta$ defines an open immersion of a local neighbourhood of
  $([I],0)$ into $\hilb^{pts}(\bbA^4)$.  In particular, if $S/I
  \not\simeq \kk$, then all components of $\hilb^{pts}(\bbA^4)$
  containing $[I]$ are elementary.
\end{theorem}



\begin{corollary}
  \label{cor:vanishing-T2}
  Suppose $I$ is supported at the origin with trivial negative
  tangents, $A=S/I$, and $T^2(A/\kk,A)_{\ge 0}=0$.  Then $I$ defines a
  smooth point on $\hilb^{pts}(\bbA^4)$.
\end{corollary}

\begin{proof}
  By Theorem \ref{thm:Jel}, it suffices to show that $[I]$ defines a
  smooth point of $\hilb^+_{pts}(\bbA^4)$.  By \cite[Theorem
    4.2]{Jelisiejew--2019}, the obstruction space for
  $(\hilb^+_{pts}(\bbA^4), [I])$ is given by $T^2(A/\kk,A)_{\ge 0}$,
  which vanishes by assumption.  Therefore, obstructions to all
  higher-order deformations vanish, showing smoothness of $[I]$ in
  $\hilb^+_{pts}(\bbA^4)$.
\end{proof}


\section{Trivial Negative Tangents, I}
\label{sec:tnt1}

Our goal in this section is to understand the tangent space of the
point $[I] \in \hilb^{pts}(\bbA^4)$ corresponding to an ideal of the
form
\[
I := \llrr{x,y}^{n_1} + \llrr{z,w}^{n_2} + \llrr{xz-yw} \subset S :=
\kk[x,y,z,w],
\]
for some $n_1, n_2 \ge 2$.
The ideal $I$ is $\fmm$-primary,
where $\fmm := \llrr{x,y,z,w} \subset S$ is the ideal of the origin $0
\in \bbA^4$.
We will prove:

\begin{proposition}
  \label{prop:Negative-tangents-for-I}
  The ideal $I$ has trivial negative tangents, hence
  every irreducible component of $\hilb^{pts}(\bbA^4)$ containing
  $[I]$ must be elementary by \cite[Theorem 1.2]{Jelisiejew--2019}.
\end{proposition}

Let $\varphi \in \Hom_S(I, S/I) \cong T_{[I]} \hilb^{pts}(\bbA^4)$, so
that $\varphi$ is determined by its values on the generators
\[
x^{n_1}, x^{n_1-1}y, \dotsc, y^{n_1}, z^{n_2}, z^{n_2-1}w, \dotsc,
w^{n_2}, xz-yw
\]
of $I$.  These values lie in the $\kk$-vector space $A := S/I$ spanned
by the cosets $x^{u_1}y^{u_2}z^{u_3}w^{u_4} + I$, where $u_1+u_2 <
n_1$ and $u_3+u_4 < n_2$; a basis of $S/I$ is obtained by ignoring any
such monomial divisible by $yw$, that is, setting
\[
\cB := \{ x^{u_1}y^{u_2}z^{u_3}w^{u_4} + I \mid u_1+u_2 <
n_1,\ u_3+u_4< n_2,\ u_2u_4 = 0 \},
\]
yields a monomial basis
for $S/I$. 
In order for $\varphi$ to be $S$-linear, it must vanish on the
syzygies of $I$, i.e.\ the following relations must hold:
\begin{align}
  & y \varphi(x^{n_1-k}y^k) = x \varphi(x^{n_1-k-1}y^{k+1}),
  &\text{for all } 0 \le k < n_1, \label{syz:xy} \\
  & w \varphi(z^{n_2-\ell}w^{\ell}) = z
  \varphi(z^{n_2-\ell-1}w^{\ell+1}), &\text{for all } 0 \le \ell <
  n_2, \label{syz:zw} \\
  & x^{n_1-1-k}y^k \varphi(xz-yw) = z \varphi(x^{n_1-k}y^k) - w
  \varphi(x^{n_1-1-k}y^{k+1}), &\text{for all } 0 \le k <
  n_1, \label{syz:qxy} \\
  & z^{n_2-1-\ell}w^{\ell} \varphi(xz-yw) = x
  \varphi(z^{n_2-\ell}w^{\ell}) - y \varphi(z^{n_2-1-\ell}w^{\ell+1}),
  &\text{for all } 0 \le \ell < n_2. \label{syz:qzw}
\end{align}

Our proof of Proposition \ref{prop:Negative-tangents-for-I} will
proceed as follows.  To prove that the tangent space $\Hom_S(I, S/I)$
vanishes in degrees at most $-2$, we will (essentially) only need to
use relations (\ref{syz:xy}) and (\ref{syz:zw}).  Then, to see that
$\Hom_S(I, S/I)_{-1}$ is exactly the $\kk$-span of the trivial tangent
vectors, we will rely on relations (\ref{syz:qxy}) and
(\ref{syz:qzw}).

\subsection{Preliminary lemmas}
\label{subsec:prelimLs-tngt-sp-dim}

We collect several helpful lemmas that will be used throughout this
paper. Given $p\in S/I$, we may expand it in the basis $\cB$. We refer
to the \emph{\bfseries support} of $p$ as the set of basis elements
with nonzero coefficients showing up in the expansion of $p$---the
support of $0$ is $\varnothing$. Note that
$\Ann_{x,y}:=\Ann(x)=\Ann(y)$ and $\Ann_{z,w}:=\Ann(z)=\Ann(w)$ are
spanned by basis vectors, so it makes sense to say whether the support
of $p$ intersects $\Ann_{x,y}$ or $\Ann_{z,w}$. Note also that if
$p,q\in S/I$ have disjoint support, then $p=q$ forces $p=q=0$.

\begin{lemma}
  \label{l:decomposition-of-polynomials}
  If $p\in S/I$, then it may be decomposed as 
  \[
  p=\sum_{0<i<n_1}y^ip_{i,0}+\sum_{0<j<n_2}w^jp_{0,j}+p_{0,0},
  \]
  where 
  \begin{enumerate}
  \item each $p_{i,j}$ is a polynomial in $x,z$,
  \item for $i\geq0$, $p_{i,0}$ has $x$-degree less than $n_1-i$ and
    $z$-degree less than $n_2$,
  \item for $j\geq0$, $p_{0,j}$ has $x$-degree less than $n_1$ and
    $z$-degree less than $n_2-j$, and
  \item all of the terms in the sum have disjoint support.
  \end{enumerate}
  Furthermore, we have
  \[
  \dim_\kk(S/I) =d(n_1,n_2) =\frac{n_1n_2}{2}(n_1+n_2).
  \]
\end{lemma}

\begin{proof}
  Expressing $p$ as a linear combination of elements of $\cB$ and
  grouping basis vectors by their $y$- and $w$-exponents, we obtain
  our desired decomposition of $p$ with properties (i)--(iv).

  Because $p_{i,0}$ has $(n_1-i)n_2$ monomials in $x$ and $z$, and
  $p_{0,j}$ has $n_1(n_2-j)$ monomials in $x$ and $z$, we see $S/I$
  has dimension
  \begin{align*}
    \bigl( (n_1-1)n_2 + (n_1-2)n_2 + \dotsb + n_2 \bigr) + (n_1n_2) &+
    \bigl( (n_2-1)n_1 + (n_2-2)n_1 + \dotsb + n_1 \bigr) \\
    &=\frac{(n_1-1)n_1}{2}n_2 + n_1n_2 + n_1\frac{(n_2-1)n_2}{2} \\
    &= \frac{n_1n_2}{2}(n_1+n_2).\qedhere
  \end{align*}
\end{proof}

\begin{lemma}
  \label{l:mindingPsAndQs}
  If $p,q\in S/I$ satisfy
  \[
  yp=xq,
  \]
  then we may write
  \[
  p=p'+xr_y+wr_w,\quad\quad q=q'+yr_y+zr_w
  \]
  such that
  \begin{enumerate}
  \item $r_y$ is a polynomial in $x,y,z$ and $r_w$ is a polynomial in
    $x,z,w$,
  \item $p',q'\in \Ann_{x,y}$, while $p-p'$ and $q-q'$ are supported
    away from $\Ann_{x,y}$,
  \item $p'$, $xr_y$, and $wr_w$, have mutually disjoint support, and
  \item $q'$, $yr_y$, and $zr_w$ have mutually disjoint support.
  \end{enumerate}
  Furthermore, $p'$ and $q'$ are unique, the image of $r_y$ in
  $S/\Ann_{x,y}$ is unique, and the image of $r_w$ in $S/\Ann_{z,w}$
  is unique.
\end{lemma}

\begin{proof}
  Expanding $p$ in the basis $\cB$, let $p'$ be the sum of all
  monomial terms of $p$ which are annihilated by $x$ (equivalently
  $y$).  This gives a decomposition analogous to Lemma
  \ref{l:decomposition-of-polynomials}, where we write
  \[
  p=p'+\sum_{0<i<n_1-1}y^ip_{i,0}+\sum_{0<j<n_2}w^jp_{0,j}+p_{0,0}
  \]
  and 
  \[
  q=q'+\sum_{0<i<n_1-1}y^iq_{i,0}+\sum_{0<j<n_2}w^jq_{0,j}+q_{0,0}
  \]
  and where no monomial terms of any of the terms $y^ip_{i,0}$,
  $w^jp_{0,j}$, $p_{0,0}$, $y^iq_{i,0}$, $w^jq_{0,j}$, $q_{0,0}$ are
  annihilated by $y$ (equivalently $x$). Then
  \[
  yp=\sum_{0<i<n_1-1}y^{i+1}p_{i,0}+\sum_{0<j<n_2}w^{j-1}xzp_{0,j}+yp_{0,0}
  \]
  and 
  \[
  xq=\sum_{0<i<n_1-1}y^ixq_{i,0}+\sum_{0<j<n_2}w^jxq_{0,j}+xq_{0,0}.
  \]
  Since none of the terms in the sum are zero (by hypothesis),
  equating terms with the same $y^iw^j$-powers, we see
  \[
  \begin{cases}
    \begin{aligned}
      p_{i,0} =xq_{i+1,0} && 0\leq i\leq n_1-3,\\
      p_{n_1-2,0} = 0, &&\\
      zp_{0,j+1} =q_{0,j} && 0\leq j\leq n_2-2,\\
      q_{0,n_2-1} =0. && 
    \end{aligned}
  \end{cases}
  \]
  Therefore, letting
  \[
  r_y=\sum_{0\leq i\leq n_1-3}y^iq_{i+1,0} \quad\textrm{and}\quad
  r_w=\sum_{0\leq j\leq n_2-2}w^jp_{0,j+1}
  \]
  we have the desired decompositions $p=p'+xr_y+wr_w$ and
  $q=q'+yr_y+zr_w$.

  It remains to prove the uniqueness properties. First, since $xr_y$
  and $wr_w$ have supports disjoint from $\Ann_{x,y}$, we see $p'$ is
  uniquely determined. Now suppose we have different choices $r'_y$
  and $r'_w$ with properties (i)--(iv). Since $p'$ is uniquely
  determined, we have
  \[
  xr_y+wr_w=xr'_y+wr'_w.
  \]
  Let $r_y=r_{y,0}+ys_y$ and $r'_y=r'_{y,0}+ys'_y$, where $r_{y,0}$
  and $r'_{y,0}$ have no $y$-terms. Expanding, we have
  \[
  xr_{y,0}+xys_y+wr_w=xr'_{y,0}+xys'_y+wr'_w.
  \]
  Then collecting terms with $y^0w^0$-powers, we see
  $xr_{y,0}=xr'_{y,0}$, so the image of $r_{y,0}$ in $S/\Ann_{x,y}$ is
  uniquely determined. Similarly, collecting terms with
  $y^iw^0$-powers for $i>0$, we have $xys_y=xys'_y$, so the image of
  $ys_y$ in $S/\Ann_{x,y}$ is also uniquely determined. Therefore, the
  image of $r_y$ in $S/\Ann_{x,y}$ is uniquely determined. Finally,
  collecting terms with $y^0w^j$-powers for $j>0$, we see
  $wr_w=wr'_w$, so the image of $r_w$ in $S/\Ann_{z,w}$ is uniquely
  determined.
\end{proof}

More generally, we have the following result.

\begin{corollary}
  \label{cor:mindingPsAndQsInGeneral}
  Let $p_0, p_1, \dotsc, p_n\in S/I$ satisfy the property
  \[
  yp_k=xp_{k+1},
  \]
  for all $0\leq k<n$, where $n \le n_1$. Then there exist $t_0, t_1,
  \dotsc, t_n \in S/I$ such that we may write
  \[
  p_k = p'_k+\sum_{i=0}^kx^{n-k}y^{k-i}z^it_i +
  \sum_{i=k+1}^nx^{n-i}z^kw^{i-k}t_i,
  \]
  for all $0 \le k \le n$, where
  \begin{enumerate}
  \item $p'_k\in \Ann_{x,y}$ and $p_k-p'_k$ is supported away from
    $\Ann_{x,y}$,
  \item $t_0$ is a polynomial in $x,y,z$ and $t_n$ is a polynomial in
    $x,z,w$,
  \item $t_i$ is a polynomial in $x,z$ for $0<i<n$,
  \item
    if $(j_{1},j_{2})$ denotes the bidegree of any element in the
    support of $t_i$, then the bounds $0\le j_1<n_1-1-n+i$ and $0\le
    j_{2}<n_2-i$ both hold, and
  \item for every $k$, all terms $p'_k$,
    $\{x^{n-k}y^{k-i}z^it_i\}_{0\leq i\leq k}$, and
    $\{x^{n-i}z^kw^{i-k}t_i\}_{k<i\leq n}$ have mutually disjoint
    support.
  \end{enumerate}
\end{corollary}

\begin{proof}
  Lemma \ref{l:mindingPsAndQs} handles the case when $n=1$.  When
  $n=2=n_1$, applying Lemma \ref{l:mindingPsAndQs} to the pairs
  $p_0,p_1$ and $p_1,p_2$ yields
  \[
  p_0 = p_0' +wr_w, \qquad p_1 = p_1' + zr_w = p_1' +w\rho_w, \qquad
  p_2 = p_2' +z\rho_w;
  \]
  here $r_y = \rho_y = 0$ follows from Lemma
  \ref{l:mindingPsAndQs}(ii) and $r_w$, $\rho_w$ are polynomials in
  $z,w$.  Write $r_w = r_{w,0} + wr_{w,+}$, where $r_{w,0}$ is a
  polynomial in $z$ and $r_{w,+}$ is a polynomial in $z,w$.  Comparing
  $w$-terms in $p_1$, we find
  \[
  zr_{w,0} = 0 \qquad\text{and}\qquad wzr_{w,+} = w\rho_w.
  \]
  As $zr_{w,0}=0$ if and only if $wr_{w,0}=0$, we may assume $r_{w,0}
  = 0$.  Let $u_0 = u_1 = 0$ and $u_2 = r_{w,+}$.  This implies
  \[
  p_0 = p_0' +w^2u_2, \qquad p_1 = p_1' + zwu_2, \qquad p_2 = p_2'
  +z^2u_2,
  \]
  giving the desired expressions; properties (i)--(v) can easily be
  verified in this case.  (If in addition $n_2 = 2$, then we take
  $u_2=0$.)  A similar proof works whenever $n=2$, where $r_y, \rho_y
  \neq 0$ are allowed if $n_1>2$.  In these cases, starting with
  \[
  p_0 = p_0' +xr_y +wr_w, \qquad p_1 = p_1' +yr_y +zr_w = p_1'
  +x\rho_y +w\rho_w, \qquad p_2 = p_2' +y\rho_y +z\rho_w
  \]
  and additionally writing $\rho_y = \rho_{y,0} + y\rho_{y,+}$, we
  find the desired expressions
  \begin{align*}
    p_0 &= p_0' +x^2u_0 +xwu_1 +w^2u_2, \\ p_1 &= p_1' +xyu_0 +xzu_1
    +zwu_2, \\ p_2 &= p_2' +y^2u_0 +yzu_1 +z^2u_2.
  \end{align*}
  Here we take $u_i=0$ if the resulting term would be $0$ or land in
  $\Ann_{x,y}$.
  
  Now assume $n>2$. Considering the tuples $(p_0,\dots,p_{n-1})$ and
  $(p_1,\dots,p_n)$, by induction, we have $t_i$ and $\tau_i$
  satisfying properties (i)--(v) (with $n$ replaced by $n-1$) and such
  that
  \[
  p_k=p'_k+\sum_{i=0}^kx^{n-1-k}y^{k-i}z^it_i +
  \sum_{i=k+1}^{n-1}x^{n-1-i}z^kw^{i-k}t_i
  \]
  for $k<n$, and
  \[
  p_k = p'_k + \sum_{i=0}^{k-1}x^{n-k}y^{k-i-1}z^i\tau_i +
  \sum_{i=k}^{n-1}x^{n-1-i}z^{k-1}w^{i-k+1}\tau_i
  \]
  for $k>0$. Since the basis vectors appearing in $p_k$ that are in
  the support of $\Ann_{x,y}$ are uniquely determined, the $p'_k$
  terms are the same in the two expressions for $p_k$.
  


  For each $1 < i < n-1$, comparing the two expressions for the
  $y^{k-i}$- or $w^{i-k}$-terms, we have
  \begin{equation}
    \label{workingOutBidgs1}
    x^{n-1-k}y^{k-i}z^it_i = x^{n-k}y^{k-i}z^{i-1}\tau_{i-1},
    \quad\text{ if } i\le k,
  \end{equation}
  \begin{equation}
    \label{workingOutBidgs2}
    x^{n-1-i}z^kw^{i-k}t_i = x^{n-i}z^{k-1}w^{i-k}\tau_{i-1},
    \quad\text{ if } i> k,
  \end{equation}
  for all $0<k<n$.  
  By our inductive assumption on the bidegrees of $t_i$ and
  $\tau_{i-1}$, if $i<n_2$, then none of the terms appearing in
  \eqref{workingOutBidgs1} or \eqref{workingOutBidgs2} is zero, and
  hence, there exists a polynomial $u_i$ in $x,z$ such that
  \[
  t_i=xu_i\quad\textrm{and}\quad \tau_{i-1}=zu_i;
  \]
  observe that any monomial in the support of $u_i$ has bidegree
  $(j_1,j_2)$ with $0\le j_1<n_1-1-n+i$ and $0\le j_2<n_2-i$. If, on
  the other hand, $i>n_2$, then both $t_i$ and $\tau_{i-1}$ vanish, so
  we may take $u_i=0$. Finally, if $i=n_2$, then $t_i=0$; by our
  assumption on the bidegree of $\tau_{i-1}$, the only way
  \eqref{workingOutBidgs1} or \eqref{workingOutBidgs2} can hold is if
  $\tau_{i-1}=0$ as well, so we may take $u_i=0$.

  Let
  \[
  \tau_0=\tau_{0,0}+y\tau_{0,+},
  \]
  where $\tau_{0,0}$ is a polynomial in $x,z$ and $\tau_{0,+}$ is a
  polynomial in $x,y,z$.  Similarly, let
  \[
  t_{n-1}=t_{n-1,0}+wt_{n-1,+},
  \]
  where $t_{n-1,0}$ is a polynomial in $x,z$ and $t_{n-1,+}$ is a
  polynomial in $x,z,w$.

  Next, by comparing the $y^0w^0$-terms in the expression for $p_1$,
  we see $x^{n-2}zt_1=x^{n-1}\tau_{0,0}$.  If $n=n_1$, then we take
  $u_1 = 0$.  Otherwise, by our assumptions on the bidegrees of $t_1$
  and $\tau_0$, no terms in the two sides of the equation are zero, so
  there exists a polynomial $u_1$ in $x,z$ such that
  \[
  t_1=xu_1\quad\textrm{and}\quad \tau_{0,0}=zu_1.
  \]
  We then see that any monomial in the support of $u_1$ has bidegree
  $(j_1,j_2)$ with $0\le j_1<n_1-n$ and $0\le j_2<n_2-1$.  Comparing
  the terms in $p_1$ with a power of $y$, we see
  \[
  x^{n-2}yt_0=x^{n-1}y\tau_{0,+}.
  \]
  Since $\Ann(x)=\Ann(y)$, this implies
  \[
  x^{n-1-\ell}y^\ell t_0=x^{n-\ell}y^\ell \tau_{0,+},
  \]
  for all $0\leq\ell\leq n-1$. If $n \ge n_1-1$, then we take $u_0 =
  0$. Otherwise, any monomial in the support of $\tau_{0,+}$ has
  bidegree $(j_1,j_2)$ with $0\le j_1<n_1-1-n$ and $0\le j_2<n_2$.
  
  Next, comparing the $w^{n-2}$-terms in $p_1$ yields
  $zw^{n-2}t_{n-1,0}=w^{n-2}x\tau_{n-2}$. Arguing in the same manner
  as we did with \eqref{workingOutBidgs1} and
  \eqref{workingOutBidgs2}, we see
  \[
  t_{n-1,0}=xu_{n-1} \quad\textrm{and}\quad \tau_{n-2}=zu_{n-1},
  \]
  for some polynomial $u_{n-1}$ in $x,z$ where $u_{n-1}=0$, if
  $n-1\geq n_2$. Comparing the $w^j$-terms in $p_1$ with $j\geq n-1$,
  we have $zw^{n-1}t_{n-1,+}=w^{n-1}\tau_{n-1}$, and since
  $\Ann(z)=\Ann(w)$, we have
  \[
  z^{\ell+1}w^{n-1-\ell}t_{n-1,+}=w^{n-1-\ell}z^\ell \tau_{n-1},
  \]
  for $0\leq\ell\leq n-1$.

  Let
  \[
  u_0=\tau_{0,+} \quad\textrm{and}\quad u_n=t_{n-1,+}.
  \]
  For $k<n$, using that $t_i=xu_i$ for $1\leq
  i\leq n-2$, we see
  \[
  p_k =p'_k+x^{n-1-k}y^kt_0 + \sum_{i=1}^k x^{n-k}y^{k-i}z^iu_i +
  \sum_{i=k+1}^{n-2}x^{n-i}z^kw^{i-k}u_i+z^kw^{n-1-k}t_{n-1}.
  \]
  Next, we have
  \[
  z^kw^{n-1-k}t_{n-1} = z^kw^{n-1-k}t_{n-1,0}+z^kw^{n-k}t_{n-1,+} =
  z^kw^{n-1-k}xu_{n-1}+z^kw^{n-k}u_n.
  \]
  Combining this with the fact that $x^{n-1-k}y^kt_0 =
  x^{n-k}y^k\tau_{0,+} = x^{n-k}y^ku_0$, we see
  \[
  p_k = p'_k + \sum_{i=0}^k x^{n-k}y^{k-i}z^iu_i + \sum_{i=k+1}^n
  x^{n-i}z^kw^{i-k}u_i,
  \]
  which is the desired expression for $p_k$ with $k<n$.

  For $k=n$, we have
  \begin{align*}
    p_n &=p'_n+\sum_{i=0}^{n-1}y^{n-1-i}z^i\tau_i\\ 
    &=p'_n+y^{n-1}(\tau_{0,0}+y\tau_{0,+}) +
    \sum_{i=1}^{n-1}y^{n-1-i}z^i(zu_{i+1})\\    
    &=p'_n+y^{n-1}(zu_1+yu_0)+\sum_{i=2}^{n}y^{n-i}z^iu_i\\
    &=p'_n+\sum_{i=0}^{n}y^{n-i}z^iu_i,
  \end{align*}
  which is the desired expression.
  
  We have now shown that $u_0, u_1, \dotsc, u_n$ satisfy properties
  (i)--(iv). For (v), let $0\le k<n$. If $0\leq i\leq k$, then
  $x^{n-k}y^{k-i}z^iu_i$ and $x^{n-1-k}y^{k-i}z^it_i$ have the same
  support, and if $i>k$, then $x^{n-i}z^kw^{i-k}u_i$ and
  $x^{n-1-i}z^kw^{i-k}t_i$ have the same support; furthermore, the
  support of $z^{k}w^{n-1-k}t_{n-1}$ is partitioned into the supports
  of $xz^kw^{n-1-k}u_{n-1}$ and $z^kw^{n-k}u_n$ for $k<n-1$, and
  similarly for $k=n-1$.  When $k=n$, $y^{n-i}z^iu_i$ and
  $y^{n-1-i}z^i\tau_{i-1}$ have the same support for $i\ge 2$, while
  the support of $y^{n-1}\tau_0$ partitions into the support of
  $y^{n-1}zu_1$ and the support of $y^nu_0$.  Hence, property (v)
  holds too.
\end{proof}

\subsection{Proof of Proposition \ref{prop:Negative-tangents-for-I}}
\label{subsec:propPf}

Recall that $\varphi \in \Hom_S(I,S/I)$ is a tangent vector.  The
\emph{\bfseries trivial tangents} are the tangent vectors
corresponding to the homomorphisms $\partial_x, \partial_y,
\partial_z, \partial_w$, where
\[
\partial_x(f) := \frac{\partial f}{\partial x} + I, \text{ for } f \in I,
\]
and $\partial_y, \partial_z, \partial_w$ are defined analogously.  As
$I$ is homogeneous, the module $\Hom_S(I, S/I)$ inherits the grading
(and the bigrading), so that
\[
\Hom_S(I, S/I) = \bigoplus_{i \in \ZZ} \Hom_S(I, S/I)_i
\]
with $\Hom_S(I, S/I)_i = \{ \varphi \in \Hom_S(I, S/I) \mid
\varphi(I_j) \subseteq (S/I)_{i+j}, \text{ for all } j \in \NN \}$
(and similarly for the bigrading).  The trivial tangents have degree
$-1$ in the standard grading.

Let us assume that $\varphi$ is graded of degree $j < 0$.  Let
\[
p_k := \varphi(x^{n_1-k}y^k),
\]
for all $0 \le k \le n_1$.  Relation \eqref{syz:xy} says that $yp_k =
xp_{k+1}$, for all $0 \le k < n_1$. Corollary
\ref{cor:mindingPsAndQsInGeneral} then applies to these values of
$\varphi$, with $n = n_1$, giving expressions
\begin{align*}
  p_0 &= p'_0 + \sum_{i=2}^{n_1} x^{n_1-i}z^kw^{i-k}t_i
  \quad\text{and} \\
  p_k &= p'_k + \sum_{i=2}^k x^{n_1-k}y^{k-i}z^it_i +
  \sum_{i=k+1}^{n_1} x^{n_1-i}z^kw^{i-k}t_i, \quad\text{for } 0<k\le
  n_1
\end{align*}
($i=0$ gives a zero term and $i=1$ gives a term in $\Ann_{x,y}$).
Observe that the degrees of all of the terms $x^{n_1-k}y^{k-i}z^it_i$
and $x^{n_1-i}z^kw^{i-k}t_i$ equal $n_1+|t_i| \ge n_1$.  Because
$j<0$, we must then have $t_i = 0$ for all $i$, by Corollary
\ref{cor:mindingPsAndQsInGeneral}(v).  Moreover, as $p_k' \in
\Ann_{x,y}$, any nonzero term of $p_k'$ must have degree at least
$n_1-1$.  This implies $p_k'$, and hence $p_k$, is zero, if $j<-1$.
By symmetry, $\varphi(z^{n_2-\ell}w^{\ell})$ is also zero, if $j<-1$.
Relation \eqref{syz:qxy} then implies that $\varphi(xz-yw) = 0$, if
$j<-1$.  This shows that $\Hom_S(I,S/I)_{j} = 0$, for $j<-1$.

Suppose that $j = -1$.  We still know that all $t_i = 0$ and so $p_k =
p_k'$, for all $k$.  Thus, we now have expressions
\[
p_k = \sum_{0 \le i < n_1} a^{(k)}_i x^{n_1-1-i}y^i +I,
\]
where each $a^{(k)}_i \in \kk$, by Lemma
\ref{l:decomposition-of-polynomials}.  Also, we have
\[
\varphi(xz-yw) = c_0x + c_1y + c_3z + c_4w +I,
\]
where all $c_i \in \kk$.

Proposition \ref{prop:Negative-tangents-for-I} now reduces to the
following:

\begin{proposition}
  \label{prop:tnt1}
  Any $S$-linear map $\varphi \colon I \to S/I$ of degree $-1$ is a
  $\kk$-linear combination of the trivial tangents $\partial_x,
  \partial_y, \partial_z, \partial_w$.
\end{proposition}

\begin{proof}
  Relation \eqref{syz:qxy} can now be written
  \begin{align*}
    c_{3} x^{n_1-1-k}y^k z + c_{4} x^{n_1-1-k}y^k w + I &= \sum_{0 \le
      i < n_1} a_{i}^{(k)} x^{n_1-1-i} y^{i} z - \sum_{0\le i<n_1}
    a_{i}^{(k+1)} x^{n_1-1-i} y^{i} w + I \\
    &= \sum_{0\le i< n_1} a_{i}^{(k)} x^{n_1-1-i} y^{i} z -
    a_{0}^{(k+1)} x^{n_1-1} w \\
    &\qquad \qquad \qquad \qquad - \sum_{0<i<n_1} a_{i}^{(k+1)}
    x^{n_1-i} y^{i-1} z + I,
  \end{align*}
  as $yw + I = xz + I$.  

  When $k = 0$, this becomes
  \begin{align*}
    c_{3} x^{n_1-1} z + c_{4} x^{n_1-1} w + I &= \sum_{0\le i<n_1}
    a_{i}^{(0)} x^{n_1-1-i} y^{i} z - a_{0}^{(1)} x^{n_1-1} w \\
    &\qquad \qquad \qquad \qquad - \sum_{0<i<n_1} a_{i}^{(1)}
    x^{n_1-i} y^{i-1} z + I,
  \end{align*}
  which implies $c_3 = a_{0}^{(0)} - a_{1}^{(1)}$, $c_4 =
  -a_{0}^{(1)}$, $a_{n_1-1}^{(0)} = 0$, and $a_{i}^{(0)} =
  a_{i+1}^{(1)}$, for all $0<i<n_1-1$.  When $0 < k < n_1 -1$, this
  becomes
  \begin{align*}
    c_{3} x^{n_1-1-k}y^k z + c_{4} x^{n_1-k}y^{k-1} z + I &=
    \sum_{0\le i< n_1} a_{i}^{(k)} x^{n_1-1-i} y^{i} z - a_{0}^{(k+1)}
    x^{n_1-1} w \\
    &\qquad \qquad \qquad \qquad - \sum_{0<i<n_1} a_{i}^{(k+1)}
    x^{n_1-i} y^{i-1} z + I,
  \end{align*}
  which shows $c_3 = a_{k}^{(k)} - a_{k+1}^{(k+1)}$, $c_4 =
  a_{k-1}^{(k)} - a_{k}^{(k+1)}$, $a_{n_1-1}^{(k)} = a_{0}^{(k+1)} =
  0$, and $a_{i}^{(k)} = a_{i+1}^{(k+1)}$, for all the remaining
  coefficients.  And when $k = n_1 -1$, this becomes
  \begin{align*}
    c_{3} y^{n_1-1} z + c_{4} x y^{n_1-2} z + I &= \sum_{0\le i< n_1}
    a_{i}^{(n_1-1)} x^{n_1-1-i} y^{i} z - a_{0}^{(n_1)} x^{n_1-1} w \\
    &\qquad \qquad \qquad \qquad - \sum_{0<i<n_1} a_{i}^{(n_1)}
    x^{n_1-i} y^{i-1} z + I,
  \end{align*}
  showing that $c_3 = a_{n_1-1}^{(n_1-1)}$, $c_4 = a_{n_1-2}^{(n_1-1)}
  - a_{n_1-1}^{(n_1)}$, $a_{0}^{(n_1)} = 0$, and $a_{i}^{(n_1-1)} =
  a_{i+1}^{(n_1)}$, for all the remaining coefficients.  This shows
  that
  \[
  c_3 = a_{0}^{(0)} - a_{1}^{(1)} = a_{1}^{(1)} - a_{2}^{(2)} = \dotsb
  = a_{n_1-2}^{(n_1-2)} - a_{n_1-1}^{(n_1-1)} = a_{n_1-1}^{(n_1-1)}
  \]
  and
  \[
  c_4 = -a_{0}^{(1)} = a_{0}^{(1)} - a_{1}^{(2)} = \dotsb =
  a_{n_1-3}^{(n_1-2)} - a_{n_1-2}^{(n_1-1)} = a_{n_1-2}^{(n_1-1)} -
  a_{n_1-1}^{(n_1)},
  \]
  while the remaining coefficients $a_{k}^{(j)}$ vanish.  Letting $a
  := c_3$ and $a' := -c_4$, this yields
  \begin{align*}
    \varphi(x^{n_1}) &= a n_1 x^{n_1-1} + I, \\
    \varphi(x^{n_1-k}y^k) &= a (n_1-k) x^{n_1-k-1}y^k + a' k
    x^{n_1-k}y^{k-1} + I, \quad\text{for } 0<k<n_1 \\
    \varphi(y^{n_1}) &= a' n_1 y^{n_1-1} + I, \text{ and }
    \\
    \varphi(xz-yw) &= c_{1}x + c_{2}y + az - a'w + I.
  \end{align*}
  Adjusting the argument for the remaining generators of $I$ yields
  \begin{align*}
    \varphi(z^{n_2}) &= b n_2 z^{n_2-1} + I, \\
    \varphi(z^{n_2-\ell}w^{\ell}) &= b (n_2-\ell)
    z^{n_2-\ell-1}w^{\ell} + b' \ell z^{n_2-\ell}w^{\ell-1} + I,
    \quad\text{for } 0<\ell<n_1 \\
    \varphi(w^{n_2}) &= b' n_2 w^{n_2-1} + I, \text{ and }
    \\
    \varphi(xz-yw) &= bx - b'y + az - a'w + I,
  \end{align*}
  where $b := c_1$ and $b' = - c_2$.  Hence, we find $\varphi =
  a\partial_x + a'\partial_y + b\partial_z + b'\partial_w$, as
  desired.
\end{proof}




This demonstrates that $I$ only has trivial negative tangents and
finishes the proof of Proposition \ref{prop:Negative-tangents-for-I}.

\section{Vanishing Nonnegative Obstruction Spaces, I}
\label{sec:obs1}

Continuing with the notation from \S\ref{sec:tnt1}, our goal in this
section is to prove the following:

\begin{proposition}
  \label{prop:IisSmooth}
  $[I]\in \hilb^{d}(\bbA^4)$ is a smooth point.
\end{proposition}

Let $A=S/I$ and $T^2_{A} := T^2(A/\kk, A)$. By Corollary
\ref{cor:vanishing-T2}, it is enough to show $T^2_{A, \ge 0} = 0$.
Let $\cF_{\bullet}$ be a minimal free resolution of $A$ over $S$,
\[
\cF_{\bullet} \colon S \stackrel{\ d_{1}^{\cF}}{\longleftarrow}
S(-n_1)^{n_1+1} \oplus S(-n_2)^{n_2+1} \oplus S(-2)
\stackrel{\ d_{2}^{\cF}}{\longleftarrow} \cF_{2} \longleftarrow
\dotsb,
\]
and set $F := \cF_{1}$.
The truncated cotangent complex $L_{\bullet} := L_{A/\kk, \bullet}$ of
the map $\kk \to A$ has terms $L_2 = Q/\Kos$, $L_1 = F/IF = F
\otimes_S A$, and $L_0 =\Omega_{S/\kk} \otimes_S A$, where $Q = \ker
d_{1}^{\cF}$ and $\Kos \subset Q$ is the submodule of Koszul
relations.  Note that $L_2 \cong \cF_{2} / \Kos'$, where $\Kos'$ is
the preimage of $\Kos$ under $d_{2}^{\cF}$.  So $L_{\bullet}$ equals
\[ \tag{$cot_{A/\kk}$} \label{cotcplxAk}
  L_{\bullet} \colon \Omega_{S/\kk} \otimes_S A \longleftarrow F/IF
  \longleftarrow Q/\Kos = \cF_{2}/\Kos'.
\]
Observe that $\cF_{\bullet}$ inherits the bigrading, and moreover,
$L_{\bullet}$ is bigraded, as is seen from the generators and the
definition of the differential.

We denote generators of $F$ by
\[
(x^{n_1-k}y^k;),\quad (z^{n_2-\ell}w^{\ell};), \quad (q;),
\]
for $0 \le k \le n_1$, $0 \le
\ell \le n_2$, and
\[
q := xz-yw,
\]
so that $d_{1}^{\cF}(g;) = g$, for a generator $g \in I$.  Albeit odd
at first glance, this notation conveniently extends to encode
syzygies, where $(g;h)$ is used to denote a syzygy obtained from
multiplication of a generator $g$ by an element $h$. Thus, among the
generators of $\cF_{2}$ are $(x^{n_1-k}y^k; x)$ and
$(z^{n_2-\ell}w^{\ell}; z)$, for $k,\ell > 0$---these elements map to
the (minimal) syzygies
\[
y(x^{n_1-k+1}y^{k-1};) - x(x^{n_1-k}y^k;) \quad \text{ and } \quad
w(z^{n_2-\ell+1}w^{\ell-1};) - z(z^{n_2-\ell}w^{\ell};)
\]
obtained respectively by multiplying $x^{n_1-k}y^k$ by $x$, and
$z^{n_2-\ell}w^{\ell}$ by $z$.  (The ideals $\llrr{x,y}^{n_1}$ and
$\llrr{z,w}^{n_2}$ are minimally resolved (individually) by the
Eliahou--Kervaire resolution, which applies more generally to
\emph{stable} ideals and can be completely described in notation
generalizing this; see \cite[\S 28]{Peeva--2011} for details).  In
addition, $\cF_{2}$ has generators we shall denote $(q;
x^{n_1-k}y^{k-1})$ and $(q; z^{n_2-\ell}w^{\ell-1})$, for $1 \le k \le
n_1$ and $1 \le \ell \le n_2$---these map to the (minimal) syzygies
\[
z (x^{n_1-k+1}y^{k-1};) - w (x^{n_1-k}y^{k};) - x^{n_1-k}y^{k-1} (q;)
\]
and
\[
x (z^{n_2-\ell+1}w^{\ell-1};) - y (z^{n_2-\ell}w^{\ell};) -
z^{n_2-\ell}w^{\ell-1} (q;)
\]
respectively.  (Cf.\ relations (\ref{syz:xy})--(\ref{syz:qzw}).)

\begin{lemma}
  \label{lem:L2gens}
  The cotangent module $L_2$ is generated by the aforementioned
  syzygies, namely, by $(x^{n_1-k}y^k; x)$, $(z^{n_2-\ell}w^{\ell};
  z)$, $(q; x^{n_1-k}y^{k-1})$, and $(q; z^{n_2-\ell}w^{\ell-1})$, for
  appropriate $k,\ell$.
\end{lemma}

\begin{proof}
  Minimality of the Eliahou--Kervaire resolution produces the
  generators $(x^{n_1-k}y^k; x)$ and $(z^{n_2-\ell}w^{\ell}; z)$ of
  $\cF_{2}$, while the syzygies $(q; x^{n_1-k}y^{k-1})$ and $(q;
  z^{n_2-\ell}w^{\ell-1})$ are minimal between $(q;)$ and either
  $(x^{n_1-k}y^k;)$ or $(z^{n_2-\ell}w^{\ell};)$.  Finally, any
  minimal syzygies between $(x^{n_1-k}y^k;)$ and
  $(z^{n_2-\ell}w^{\ell};)$ must be Koszul relations, as the
  corresponding generators of $I$ have no variables in common.
\end{proof}

By definition, $T^2_{A}$ is the quotient of $L^2 := \Hom_{A}(L_2, A)$
by the image of $d^1_L := - \circ d_2^L$, where $d_2^L \colon L_2 \to
L_1$ is induced by $d_{2}^{\cF}$.
We aim to understand $L^2_{\ge 0}$, specifically showing the
following.

\begin{proposition}
  \label{prop:obs1}
  Given the preceding set-up, we have $L^2_{\ge 0} = d^1_L(L^1_{\ge
    0})$, i.e.\ $T^2_{A, \ge 0} = 0$.
\end{proposition}

In other words, any $A$-linear map $\psi \colon L_2 \to A$ of
nonnegative degree extends over the differential $d_2^{L} \colon L_2
\to L_1$ to a compatible $A$-linear map $\psi' \colon L_1 \to A$.
Before proving this, we set some notation and record a helpful lemma.
According to Lemma \ref{lem:L2gens}, $\psi$ is determined by its
values on $(x^{n_1-k}y^k; x)$, $(z^{n_2-\ell}w^{\ell}; z)$, $(q;
x^{n_1-k}y^{k-1})$, and $(q; z^{n_2-\ell}w^{\ell-1})$, for $1 \le k
\le n_1$ and $1 \le \ell \le n_2$.
Lemma \ref{l:decomposition-of-polynomials} yields expressions
\begin{align*}
  \psi(x^{n_1-k}y^k; x) &=: P_k = \sum_{0<i<n_1} y^i(P_k)_{i,0} +
  (P_k)_{0,0} + \sum_{0<j<n_2} w^j(P_k)_{0,j} \quad\text{ and }
  \\ \psi(q; x^{n_1-k}y^{k-1}) &=: Q_k = \sum_{0<i<n_1} y^i
  (Q_k)_{i,0} + (Q_k)_{0,0} + \sum_{0<j<n_2} w^j(Q_k)_{0,j}
\end{align*}
along with similar expressions for $\psi(z^{n_2-\ell}w^{\ell}; z)$ and
$\psi(q; z^{n_2-\ell}w^{\ell-1})$.

\begin{lemma}
  \label{lem:T2helper}
  Any homomorphism $\psi \colon L_2 \to A$ as above satisfies the
  following:
  \begin{enumerate}
  \item 
    all terms of $(P_k)_{0,j}$ are divisible by $x$, for $0 \le j <
    n_2$, and
  \item the equalities
    $x Q_k = w P_{k}$ and $y Q_k = z P_{k}$ hold.
  \end{enumerate}
  The analogous statements for $\psi(z^{n_2-\ell}w^{\ell}; z)$ and
  $\psi(q; z^{n_2-\ell}w^{\ell-1})$ are also true.
\end{lemma}

\begin{proof}
  Observe that $x^{n_1-k}y^{k-1} (x^{n_1-k}y^k; x) \in \Kos'$, so that
  \begin{align*}
  0 &= x^{n_1-k}y^{k-1} \psi(x^{n_1-k}y^k; x) =
  x^{n_1-k}y^{k-1}(P_k)_{0,0} + x^{n_1-k}y^{k-1}\sum_{0<j<n_2}
  w^j(P_k)_{0,j} \\ &= y^{k-1}x^{n_1-k}(P_k)_{0,0} + \sum_{0<j<k}
  y^{k-1-j} x^{n_1-k+j} z^j (P_k)_{0,j} + \sum_{k\le j<n_2} w^{j-k+1}
  x^{n_1-1} z^{k-1} (P_k)_{0,j} \\ &= \sum_{0<j<k} y^{j} x^{n_1-1-j}
  z^{k-1-j} (P_k)_{0,k-1-j} + x^{n_1-1} z^{k-1} (P_k)_{0,k-1} +
  \sum_{0<j\le n_2-k} w^{j} x^{n_1-1} z^{k-1} (P_k)_{0,k-1+j} \\ &=
  \sum_{0<j<k} y^{j} x^{n_1-1-j} z^{k-1-j} (P_k)_{0,k-1-j}^z +
  x^{n_1-1} z^{k-1} (P_k)_{0,k-1}^z + \sum_{0<j\le n_2-k} w^{j}
  x^{n_1-1} z^{k-1} (P_k)_{0,k-1+j}^z,
  \end{align*}
  where $(P_k)_{0,j}^z$ is the $x^0z^{\ge 0}$-part of $(P_k)_{0,j}$.
  Since the $z$-degree of $(P_k)_{0,j}^z$ is less than $n_2-j$, this
  shows $(P_k)_{0,j}^z = 0$, for all $0 \le j < n_2$, proving (i).  To
  prove the first equality in (ii), simply observe that $x(q;
  x^{n_1-k}y^{k-1}) - w(x^{n_1-k}y^{k}; x) \in \Kos'$; the second
  equality is similarly proved.  The analogous statements for
  $\psi(z^{n_2-\ell}w^{\ell}; z)$ and $\psi(q;
  z^{n_2-\ell}w^{\ell-1})$ are obtained by switching the roles of
  $x,y$ and $z,w$.
\end{proof}

Lemma \ref{lem:T2helper}(i) says that all terms of $\psi(x^{n_1-k}y^k;
x)$ are divisible by $x$ or $y$.
Part (ii) imposes further restrictions on the
values of $\psi$.  We can now proceed with the proof of Proposition
\ref{prop:obs1}.

\begin{proof}[Proof of Proposition \ref{prop:obs1}]
  The goal is to find compatible values $\psi'(x^{n_1-k}y^k;)$,
  $\psi'(z^{n_2-\ell}w^{\ell};)$, and $\psi'(q;)$ for the generators
  of $L_1$ in order to define an extension $\psi' \colon L_1 \to A$
  such that $\psi = \psi' \circ d_2^L$, i.e.~we require values
  $\pi_k := \psi'(x^{n_1-k}y^k;)$ and $\rho := \psi'(q;)$
  in $A$ such that the equalities
  \begin{align}
    P_k &= \ y\pi_{k-1} - x\pi_k \quad\text{ and } \label{eq:xy} \\
    Q_k &= z \pi_{k-1} - w \pi_{k} - x^{n_1-k}y^{k-1}
    \rho \label{eq:qxy}
  \end{align}
  hold when $k>0$, along with analogous equalities involving
  $\psi(z^{n_2-\ell}w^{\ell}; z)$, $\psi(q; z^{n_2-\ell}w^{\ell-1})$,
  $\psi'(z^{n_2-\ell}w^{\ell};)$, and $\rho$.  As $T^2_A$ is bigraded
  and graded, we simplify by assuming that $\psi$
  and $\psi'$ are bihomogeneous of bidegree $(-d_1,d_2)$ and total
  degree $d_2-d_1 \ge 0$.

  To begin, suppose that $P_k = 0$ for all $k$.  If all $Q_k = 0$,
  then \eqref{eq:xy} and \eqref{eq:qxy} are solved by setting all
  $\pi_k = 0$ and $\rho = 0$.
  
  Next, assume that some $Q_\ell \neq 0$. Lemma \ref{lem:T2helper}(ii)
  then tells us $Q_\ell$ is a nonzero element of $\Ann_{x,y}$; since
  $\bideg(Q_{\ell}) = (n_1-d_1,1+d_2)$, this forces $d_1 = 1$ and $0 <
  d_2 < n_2 -1$. We will solve \eqref{eq:qxy} by choosing $\rho=0$ and
  $\pi_0, \pi_1, \dotsc, \pi_{n_1}\in\Ann_{x,y}$; hence \eqref{eq:xy}
  is trivially satisfied. Then, for all $k$, we have
  \begin{align*}
    Q_k &= \sum_{0<i<n_1} y^i (Q_k)_{i,0} + (Q_k)_{0,0} +
    \sum_{0<j<n_1} w^j (Q_k)_{0,j} \\ &= \sum_{0<i<n_1} y^i
    x^{n_1-1-i} z^{d_2+1} b^{(k)}_{i,0} + x^{n_1-1} z^{d_2+1}
    b^{(k)}_{0,0} + \sum_{0<j\le d_2+1} w^j x^{n_1-1} z^{d_2+1-j}
    b^{(k)}_{0,j},
  \end{align*}
  where each $b^{(k)}_{i,j} \in \kk$.  Similarly, we have
  \begin{align*}
    z\pi_{k-1} &= z \left( \sum_{0<i<n_1} y^i (\pi_{k-1})_{i,0} +
    (\pi_{k-1})_{0,0} + \sum_{0<j<n_2} w^j (\pi_{k-1})_{0,j} \right)
    \\ &= z \left( \sum_{0<i<n_1} y^i x^{n_1-1-i} z^{d_2}
    \mu^{(k-1)}_{i,0} + x^{n_1-1} z^{d_2} \mu^{(k-1)}_{0,0} +
    \sum_{0<j\le d_2} w^j x^{n_1-1} z^{d_2-j} \mu^{(k-1)}_{0,j}
    \right) \\ &= \sum_{0<i<n_1} y^i x^{n_1-1-i} z^{d_2+1}
    \mu^{(k-1)}_{i,0} + x^{n_1-1} z^{d_2+1} \mu^{(k-1)}_{0,0} +
    \sum_{0<j\le d_2} w^j x^{n_1-1} z^{d_2+1-j} \mu^{(k-1)}_{0,j},
  \end{align*}
  along with
  \begin{align*}
    w\pi_{k} &= w \left( \sum_{0<i<n_1} y^i (\pi_{k})_{i,0} +
    (\pi_{k})_{0,0} + \sum_{0<j<n_2} w^j (\pi_{k})_{0,j} \right)
    \\ &= w \left( \sum_{0<i<n_1} y^i x^{n_1-1-i} z^{d_2}
    \mu^{(k)}_{i,0} + x^{n_1-1} z^{d_2} \mu^{(k)}_{0,0} + \sum_{0<j\le
      d_2} w^j x^{n_1-1} z^{d_2-j} \mu^{(k)}_{0,j} \right) \\ &=
    \sum_{0<i<n_1} y^{i-1} x^{n_1-i} z^{d_2+1} \mu^{(k)}_{i,0} + w
    x^{n_1-1} z^{d_2} \mu^{(k)}_{0,0} + \sum_{0<j\le d_2} w^{j+1}
    x^{n_1-1} z^{d_2-j} \mu^{(k)}_{0,j} \\ &= \sum_{0<i<n_1-1} y^{i}
    x^{n_1-1-i} z^{d_2+1} \mu^{(k)}_{i+1,0} + x^{n_1-1} z^{d_2+1}
    \mu^{(k)}_{1,0} + \sum_{0<j\le d_2+1} w^{j} x^{n_1-1} z^{d_2+1-j}
    \mu^{(k)}_{0,j-1},
  \end{align*}
  where each $\mu^{(k)}_{i,j} \in \kk$.  Thus, \eqref{eq:qxy} reduces
  to the system
  \[ \tag{$\ddagger$} \label{ddagger}
  \begin{cases}
    \begin{aligned}
      b^{(k)}_{n_1-1,0} &= \mu^{(k-1)}_{n_1-1,0} &&\text{ if
      } i = n_1-1, \\
      b^{(k)}_{i,0} &= \mu^{(k-1)}_{i,0} - \mu^{(k)}_{i+1,0}
      &&\text{ if } 0 \le i \le n_1-2, \\
      b^{(k)}_{0,j} &= \mu^{(k-1)}_{0,j} - \mu^{(k)}_{0,j-1}
      &&\text{ if } 0 < j \le d_2, \text{ and } \\
      b^{(k)}_{0,d_2+1} &= - \mu^{(k)}_{0,d_2} &&\text{ if } j = d_2+1.
    \end{aligned}
  \end{cases}
  \]
  This gives a system of linear equations in the variables
  $\mu_{i,j}^{(k)}$ which splits into two independent subsystems:
  \begin{enumerate}
  \item all equations involving $\mu_{i,j}^{(k)}$'s with $k -i +j \le
    d_2$,
  \item all equations involving $\mu_{i,j}^{(k)}$'s with $k -i +j >
    d_2$
  \end{enumerate}
  (the quantity $k -i +j$ is constant among $\mu_{i,j}^{(k)}$'s in
  each equation).  If $i = n_1-1$, then $k - i \le d_2$ holds, because
  $k \le n_1$ and $d_2 \ge 1$, so the first equation of
  \eqref{ddagger} belongs to (i).  Also, $j = d_2$ implies $k + j >
  d_2$ exactly when $k > 0$, so the fourth equation of \eqref{ddagger}
  belongs to (ii).  For (i), after fixing the $\mu_{i,0}^{(n_1)}$ and
  $\mu_{0,j}^{(n_1)}$ arbitrarily, there is a unique solution given by
  \begin{align*}
    \mu_{i,0}^{(k)} &=
    \begin{cases}
      b_{i,0}^{(k+1)} + b_{i+1,0}^{(k+2)} + \dots
      +b_{n_1-1,0}^{(k+n_1-i)} & \text{if }i\geq k,\\ b_{i,0}^{(k+1)}
      +b_{i+1,0}^{(k+2)} +\dots
      +b_{i+n_1-k-1,0}^{(n_1)}+\mu_{i+n_1-k,0}^{(n_1)} & \text{if
      }i<k,
    \end{cases} \\
    \mu_{0,j}^{(k)} &=
    \begin{cases}
      b_{0,j}^{(k+1)} + b_{0,j-1}^{(k+2)} + \dotsb +
      b_{0,j-n_1+k+1}^{(n_1)} + \mu_{0,j-n_1+k}^{(n_1)} &\text{if } j
      \ge n_1-k, \\ b_{0,j}^{(k+1)} +
      \dotsb + b_{0,1}^{(k+j)} + b_{0,0}^{(k+j+1)} + b_{1,0}^{(k+j+2)}
      \dotsb + b_{n_1-1-k-j,0}^{(n_1)} + \mu_{n_1-k-j,0}^{(n_1)}
      &\text{if } j < n_1-k.
    \end{cases}
  \end{align*}
  For (ii), there is a unique solution given by
  \begin{align*}
    \mu_{i,0}^{(k)} &= -b_{i-1,0}^{(k)} -b_{i-2,0}^{(k-1)} - \dotsb
    -b_{0,0}^{(k-i+1)} -b_{0,1}^{(k-i)} -\dotsb
    -b_{0,d_2+1}^{(k-i-d_2)} \quad \text{ and } \\
    \mu_{0,j}^{(k)} &= - b_{0,j+1}^{(k)} - b_{0,j+2}^{(k-1)} - \dotsb
    -b_{0,d_2+1}^{(k+j-d_2)}.
  \end{align*}
  This proves that \eqref{ddagger}, and thus \eqref{eq:qxy}, can be
  solved under the assumption that all $P_k = 0$, and therefore that
  \eqref{eq:xy} and \eqref{eq:qxy} can be solved under this
  assumption.

  Lastly, we turn to the case where some $P_\ell \neq 0$. Since
  $\bideg(P_\ell) = (n_1+1-d_1, d_2)$, we must have $d_1 > 1$ and
  $d_2<n_2$. Applying Lemma \ref{lem:T2helper}(i), we see $0 <
  n_1+1-d_1<n_1$. Now, for all $k$, we have
  \begin{align*}
    P_k &= \sum_{0<i<n_1} y^i (P_k)_{i,0} + (P_k)_{0,0} +
    \sum_{0<j<n_1} w^j (P_k)_{0,j} \\ &= \sum_{0<i\le n_1-d_1+1} y^i
    x^{n_1+1-d_1-i} z^{d_2} c^{(k)}_{i,0} + x^{n_1+1-d_1} z^{d_2}
    c^{(k)}_{0,0} + \sum_{0<j\le d_2} w^j x^{n_1+1-d_1} z^{d_2-j}
    c^{(k)}_{0,j},
  \end{align*}
  where each $c^{(k)}_{i,j} \in \kk$.  Similarly, we have
  \begin{align*}
    \!\! y\pi_{k-1} &= y \left( \sum_{0<i<n_1} y^i (\pi_{k-1})_{i,0} +
    (\pi_{k-1})_{0,0} + \sum_{0<j<n_2} w^j (\pi_{k-1})_{0,j} \right)
    \\ &= y \left( \sum_{0<i\le n_1-d_1} y^i x^{n_1-d_1-i} z^{d_2}
    \lambda^{(k-1)}_{i,0} + x^{n_1-d_1} z^{d_2} \lambda^{(k-1)}_{0,0}
    + \sum_{0<j\le d_2} w^j x^{n_1-d_1} z^{d_2-j}
    \lambda^{(k-1)}_{0,j} \right) \\ &= \sum_{0<i\le n_1-d_1} y^{i+1}
    x^{n_1-d_1-i} z^{d_2} \lambda^{(k-1)}_{i,0} + yx^{n_1-d_1} z^{d_2}
    \lambda^{(k-1)}_{0,0} + \sum_{0<j\le d_2} w^{j-1} x^{n_1+1-d_1}
    z^{d_2-j+1} \lambda^{(k-1)}_{0,j} \\ &= \sum_{0<i\le n_1+1-d_1}
    y^{i} x^{n_1+1-d_1-i} z^{d_2} \lambda^{(k-1)}_{i-1,0} +
    x^{n_1+1-d_1} z^{d_2} \lambda^{(k-1)}_{0,1} + \sum_{0<j<d_2} w^{j}
    x^{n_1+1-d_1} z^{d_2-j} \lambda^{(k-1)}_{0,j+1}
  \end{align*}
  along with
  \begin{align*}
    x\pi_{k} &= x \left( \sum_{0<i<n_1} y^i (\pi_{k})_{i,0} +
    (\pi_{k})_{0,0} + \sum_{0<j<n_2} w^j (\pi_{k})_{0,j} \right) \\ &=
    x \left( \sum_{0<i\le n_1-d_1} y^i x^{n_1-d_1-i} z^{d_2}
    \lambda^{(k)}_{i,0} + x^{n_1-d_1} z^{d_2} \lambda^{(k)}_{0,0} +
    \sum_{0<j\le d_2} w^j x^{n_1-d_1} z^{d_2-j} \lambda^{(k)}_{0,j}
    \right) \\ &= \sum_{0<i\le n_1-d_1} y^i x^{n_1+1-d_1-i} z^{d_2}
    \lambda^{(k)}_{i,0} + x^{n_1+1-d_1} z^{d_2} \lambda^{(k)}_{0,0} +
    \sum_{0<j\le d_2} w^j x^{n_1+1-d_1} z^{d_2-j} \lambda^{(k)}_{0,j},
  \end{align*}
  where each $\lambda^{(k)}_{i,j} \in \kk$.  Thus, \eqref{eq:xy}
  reduces to the following system:
  \[ \tag{$\dagger$} \label{dagger}
  \begin{cases}
    \begin{aligned}
      c^{(k)}_{n_1-d_1+1,0} &= \lambda^{(k-1)}_{n_1-d_1,0} &&\text{ if
      } i = n_1-d_1+1, \\
      c^{(k)}_{i,0} &= \lambda^{(k-1)}_{i-1,0} - \lambda^{(k)}_{i,0}
      &&\text{ if } 0 < i \le n_1-d_1, \\
      c^{(k)}_{0,j} &= \lambda^{(k-1)}_{0,j+1} - \lambda^{(k)}_{0,j}
      &&\text{ if } 0 \le j < d_2, \text{ and } \\
      c^{(k)}_{0,d_2} &= - \lambda^{(k)}_{0,d_2} &&\text{ if } j = d_2.
    \end{aligned}
  \end{cases}
  \]
  This system has the same general form as \eqref{ddagger} and can be
  solved in exactly the same way.
  
  Therefore, the system \eqref{dagger}, and thus \eqref{eq:xy}, can be
  solved when some $P_\ell \neq 0$ and we may proceed to studying
  equality \eqref{eq:qxy}.  It is easily seen from the basis $\cB$
  that multiplication-by-$x$ defines an injective $\kk$-linear map
  $A_{(i_1,i_2)} \to A_{(i_1+1,i_2)}$ between bigraded pieces of $A$
  when $0 \le i_1 < n_1-1$.  But $d_1 \ge 2$, so Lemma
  \ref{lem:T2helper} gives
  \[
  xQ_k = wP_k = w(y\pi_{k-1} - x\pi_k) = xz\pi_{k-1} - xw\pi_k =
  x(z\pi_{k-1} - w\pi_k - x^{n_1-k}y^{k-1}\rho),
  \]
  which implies that $\pi_0, \pi_1, \dotsc, \pi_{n_1}$ and $\rho$
  satisfy \eqref{eq:qxy}, for any choice of $\rho$.
  
  To finish, we observe that by symmetry, the same approach solves the
  analogous equations $(\ref{eq:xy}')$ and $(\ref{eq:qxy}')$ obtained
  from $(\ref{eq:xy})$ and $(\ref{eq:qxy})$ where the roles of $x,y$
  and $z,w$ are switched. To see that the solutions we obtain are
  consistent, note that $\rho$ is the only term appearing in both sets
  of equations $(\ref{eq:xy})$, $(\ref{eq:qxy})$ and $(\ref{eq:xy}')$,
  $(\ref{eq:qxy}')$, and that in every case we can solve these
  equations with $\rho = 0$.  Hence, $\psi \colon L_2 \to A$ factors
  through a map $\psi' \colon L_1 \to A$.
\end{proof}

\begin{proof}[Proof of Proposition \ref{prop:IisSmooth}]
  Combine Proposition \ref{prop:obs1} and Corollary
  \ref{cor:vanishing-T2}.
\end{proof}

\section{Trivial Negative Tangents, II}
\label{sec:tnt2}

We continue our study of negative tangents, focusing on the family of
ideals described in Theorem \ref{thm:intro2}.  As before, let $I :=
\llrr{x,y}^{n_1} + \llrr{z,w}^{n_2} + \llrr{q}$, where $q := xz - yw$
and now $n_1, n_2 \ge 3$, and set $A := S/I$.
We begin with a simple observation.

\begin{lemma}
  \label{lem:socA}
  The socle $\Soc A$ is bigraded and equals $A_{(n_1-1,n_2-1)} =
  A_{n_1+n_2-2}$.
\end{lemma}

In other words, Lemma \ref{lem:socA} says that the socle of $A$ equals
the bidegree $(n_1-1,n_2-1)$ piece of $A$, which coincides with the
total degree $n_1+n_2-2$ piece of $A$.

\begin{proof}
  Since $\fmm_A$ is bigraded, $\Soc A$ is as well.  When $i_1 \le
  n_1-2$, multiplication-by-$x$ gives an injective map $A_{(i_1,i_2)}
  \to A_{(i_1+1,i_2)}$; when $i_2 \le n_2-2$, the
  multiplication-by-$z$ map $A_{(i_1,i_2)} \to A_{(i_1,i_2+1)}$ is
  injective.  It is clear that $x,y,z,w$ kill $A_{(n_1-1,n_2-1)}$,
  thus, we find that $\Soc A = A_{(n_1-1,n_2-1)} = A_{n_1+n_2-2}$.
\end{proof}

Let $J := I + \llrr{s}$, where $s \in S_{(n_1-1,n_2-1)} \smallsetminus
I$, and $B := S/J$, so there is a short exact sequence
\[
0 \longrightarrow J/I \longrightarrow A
\stackrel{\pi}{\longrightarrow} B \longrightarrow 0.
\]
By Proposition \ref{prop:tnt1}, we know that $I$ has trivial negative
tangents---we wish to show that $J$ has trivial negative tangents too.
We could proceed directly as in \S\ref{sec:tnt1}, performing
elementary computations; instead, we apply a standard long exact
sequence in tangent cohomology (see Remark \ref{rmk:graded} and
\cite[Theorem 3.5]{Hartshorne--2010}).  Namely, the pair of natural
ring maps $\kk \to A \to B$ leads to a long exact sequence containing
the following portion:
\begin{equation}
  \label{LES-T1}
  \dotsb \longrightarrow T^1(B/A,B) \longrightarrow T^1(B/\kk,B)
  \longrightarrow T^1(A/\kk,B) \longrightarrow \dotsb,
\end{equation}
which we use to show that $T^1(B/\kk,B)_{<0} = 0$.

\begin{lemma}
  \label{lem:LES-T1}
  We have $T^1(B/A,B)_{<0} = 0$ and $T^1(A/\kk,B)_{<0} = 0$.
\end{lemma}

\begin{proof}
  We examine $T^1(A/\kk,B)$ first, using the notation of
  \S\ref{sec:obs1}.  The truncated cotangent complex of $\kk \to A$
  is described by \eqref{cotcplxAk}.  As $L_0 \cong A(-1)^4$ in the
  standard grading, the $B$-dual is
  \[
  \Hom_A(L_{\bullet},B) \colon B(1)^4 \longrightarrow \Hom_{A}(F/IF,B)
  \longrightarrow \Hom_A(\cF_{2}/\Kos',B).
  \]

  Let $\psi \colon L_1 = F/IF \to B$ represent an element of
  $T^1(A/\kk,B)_i$ with $i < 0$; $\psi$ is determined by its values on
  the generators $(g;)$ of $L_1$ (listed just after
  \eqref{cotcplxAk}).  Because the degrees satisfy $|\psi(g;)| = |g|+i
  < n_1+n_2-2 = |s|$, the map $\pi$ identifies $A_{|\psi(g;)|}$ with
  $B_{|\psi(g;)|}$.  This allows us to define an $A$-linear map
  $\widetilde{\psi} \colon L_1 \to A$ satisfying $\psi = \pi \circ
  \widetilde{\psi}$ via $\widetilde{\psi}(g;) := \psi(g;)$, for each
  $g$. Checking degrees also shows that $\widetilde{\psi} \circ
  d_2^{L} = 0$:~for instance, $|\psi \circ d_2^L(x^{n_1-k}y^k;x)| =
  n_1+1+i < n_1+n_2-2$ implies $\widetilde{\psi} \circ
  d_2^{L}(x^{n_1-k}y^k;x) = \psi \circ d_2^{L}(x^{n_1-k}y^k;x) = 0$;
  other generators similarly vanish.  This means $\widetilde{\psi}$
  defines an element of $T^1(A/\kk,A)_i$.  We know $T^1(A/\kk,A)_{<0}
  = 0$, by Proposition \ref{prop:tnt1}, so
  $\widetilde{\psi}$ is a $\kk$-linear combination of the trivial
  tangents $\partial_x, \partial_y, \partial_z, \partial_w$.  The
  identification of $\widetilde{\psi}(g;)$ with $\psi(g;)$, for each
  $g$, then implies $\psi$ is a $\kk$-linear combination of the
  trivial tangents.  Thus, we have $T^1(A/\kk,B)_{<0} = 0$.


  We examine $T^1(B/A,B)$ next, forming the truncated cotangent
  complex of
  $\pi$.  As $\pi$ is
  surjective, we set $R_{B/A} = A$; then we may choose $F_{B/A} =
  A(-n_1-n_2+2)$ as $J/I = \llrr{s +I}$ is principal. Next, $Q_{B/A} =
  \fmm_A(-n_1-n_2+2)$ holds, because $s +I \in \Soc A$;
  this also guarantees that $\Kos_{B/A} = 0$.  Finally,
  $\Omega_{R_{B/A}/A} = \Omega_{A/A} = 0$, so we see that the
  truncated cotangent complex equals
  \[
  L_{B/A,\bullet} \colon 0 \longleftarrow B(-n_1-n_2+2) \longleftarrow
  \fmm_A(-n_1-n_2+2),
  \]
  where the differential
  $d_2^{B/A}$ is a twist of $\pi|_{\fmm_A} \colon \fmm_A \to B$.
  This implies
  \begin{align*}
    T^1(B/A,B) = \ker d_{B/A}^1 &= \{ \varphi \colon B(-n_1-n_2+2) \to
    B \mid \varphi \circ d_2^{B/A} = 0 \} \\ &= \{ \varphi \colon
    B(-n_1-n_2+2) \to B \mid \varphi|_{\fmm_B} = 0 \} \cong (\Soc
    B)(n_1+n_2-2),
  \end{align*}
  so that $T^1(B/A,B)_{<0} = 0$ holds by the following lemma.
\end{proof}

\begin{lemma}
  \label{lem:socB}
  The socle $\Soc B$ is bigraded and equals $B_{(n_1-1,n_2-1)} =
  B_{n_1+n_2-2}$.
\end{lemma}

\begin{proof}
  As $B$ and $\fmm_B$ are bigraded, so $\Soc B$ is bigraded.  When
  $i_1 \le n_1-2$ and $i_2 \le n_2-2$, the multiplication-by-$x$ and
  -$y$ maps $[x], [y] \colon B_{(i_1,i_2)} \to B_{(i_1+1,i_2)}$ and
  the multiplication maps $[z],[w] \colon B_{(i_1,i_2)} \to
  B_{(i_1,i_2+1)}$ are injective.  When $(i_1,i_2) = (n_1-2,n_2-1)$,
  the kernels of $[x]$ and $[y]$ have dimension at most $1$.
  Suppose that $b \in B_{(n_1-2,n_2-1)}$ satisfies $xb = yb = 0$;
  treating $b$ as an element of $A$, this means $xb$ and $yb$ are
  scalar multiples of $s$ and thus of each other; a direct computation
  in the basis $\cB$ then shows that $b=0$.  A similar occurrence
  holds for $[z]$ and $[w]$, when $(i_1,i_2) = (n_1-1,n_2-2)$.
  Thus, we find that $\Soc B = B_{(n_1-1,n_2-1)} = B_{n_1+n_2-2}$.
\end{proof}

This proves the following.

\begin{proposition}
  \label{prop:tnt2}
  The ideal $J$ has trivial negative tangents.
\end{proposition}

\begin{proof}
  Lemmas \ref{lem:socA}, \ref{lem:LES-T1}, and \ref{lem:socB} show
  that $T^1(B/A,B)_{<0} = 0 = T^1(A/\kk,B)_{<0}$, proving that
  $T^1(B/\kk,B)_{<0} = 0$ via the long exact sequence \eqref{LES-T1}.
\end{proof}




We show next that the above arguments
can oftentimes be iterated.  This is done after a preliminary lemma.

\begin{lemma}
  \label{l:socle-r-condition->socle-i-condition}
  Let $s_1, s_2, \dotsc, s_r \in \Soc A$ and $A^{(i)} = A / \llrr{s_1,
    s_2, \dotsc, s_i}$. If the socle of $A^{(r)}$ satisfies $\Soc
  A^{(r)} = A^{(r)}_{(n_1-1, n_2-1)}$, then $\Soc A^{(i)} =
  A^{(i)}_{(n_1-1, n_2-1)}$ holds, for all $1 \leq i \leq r$.
\end{lemma}

\begin{proof}
  We have surjections
  \[
  A = A^{(0)}\xrightarrow{\pi_0} A^{(1)}\xrightarrow{\pi_1}\dotsb
  \xrightarrow{\pi_{r-1}} A^{(r)}.
  \]
  We already know $\Soc A = A_{(n_1-1, n_2-1)}$ by Lemma
  \ref{lem:socA}. In particular, for all $i$, the image of $s_i$ in
  $A^{(i-1)}$ is contained in $A^{(i-1)}_{(n_1-1, n_2-1)}$.

  We prove the lemma by backwards induction on $i$. We have the
  following diagram
  \[
  \xymatrix{ \Soc A^{(i)}\ar[r]^-{\pi_i} & \Soc
    A^{(i+1)}\\ A^{(i)}_{(n_1-1,n_2-1)}\ar[r]\ar@{^{(}->}[u] &
    A^{(i+1)}_{(n_1-1,n_2-1)}\ar@{^{(}->}[u]\text{.}}
  \]
  If $A^{(i)}_{(n_1-1,n_2-1)} \neq \Soc A^{(i)}$, then there exists $s
  \in \Soc A^{(i)}$ with bidegree $(k,\ell) \neq (n_1-1,n_2-1)$. But
  then $\pi_i(s)\in A^{(i+1)}_{(k,\ell)}\cap\Soc A^{(i+1)} =
  0$. Therefore, $s$ is a scalar multiple of the image of $s_{i+1}$,
  which we know is in $A^{(i)}_{(n_1-1, n_2-1)}$, giving a
  contradiction.
\end{proof}

\begin{corollary}
  \label{cor:tnt2}
%
  Let $s_1, s_2, \dotsc, s_{r} \in S$,
  \[
  J' = I + \llrr{s_1, s_2, \dotsc, s_{r}}
  \]
  and $B' = S/J'$. Assume every $s_i+I\in \Soc A$ and $\Soc B' =
  B'_{(n_1-1,n_2-1)}$. Then $J'$ has trivial negative tangents.
 \end{corollary}

\begin{proof}
  We prove the result by induction on $r$. Proposition \ref{prop:tnt2}
  handles the case $r = 1$, so let $r > 1$. Let $I^{(0)} := I$,
  $I^{(i)} := I + \llrr{s_1, s_2, \dotsc, s_i}$, and
  $A^{(i)}=S/I^{(i)}$ for $1 \le i \le r$.  We may further suppose
  that $s_{i+1} + I^{(i)} \in \Soc A^{(i)}$ is nonzero. Note that by
  Lemma \ref{l:socle-r-condition->socle-i-condition}, $\Soc A^{(i)} =
  A^{(i)}_{(n_1-1, n_2-1)}$ for all $1 \leq i \leq r$.
 
  Set $I' = I^{(r-1)}$, $A' = S/I'$, and $\pi' \colon A' \to B'$.
  We use the long exact sequence of the pair of ring maps $\kk \to A'
  \to B'$, studying the portion
  \[
  \dotsb \longrightarrow T^1(B'/A',B') \longrightarrow T^1(B'/\kk,B')
  \longrightarrow T^1(A'/\kk,B') \longrightarrow \dotsb.
  \]
  To understand $T^1(A'/\kk,B')$, we use the truncated cotangent
  complex $L'_{\bullet}$ of $\kk \to A'$.  Let $\cF'_{\bullet}$ be the
  minimal free resolution of $A'$; we have
  \[
  F' := \cF'_{1} = \cF_1\oplus \bigoplus_{i = 1}^{r-1} S(s_i;),
  \]
 where $(s_i;) \mapsto s_i$ and where $\cF_{\bullet}$ is the minimal
 free resolution of $A$. Let $[\psi'] \in T^1(A'/\kk,B')_j$. We
 further assume that $\psi'$ is bigraded.
 
 First assume that $j<-1$.  Like in the proof of Lemma
 \ref{lem:LES-T1}, we wish to lift a class $[\psi'] \in
 T^1(A'/\kk,B')_j$ to $[\widetilde{\psi'}] \in T^1(A'/\kk,A')_j$.  We
 have $|\psi'(s_i;)| = n_1+n_2-2 +j < n_1+n_2-2$, so we can define a
 map $\widetilde{\psi'} \colon F'/I'F' \to A'$ via
 $\widetilde{\psi'}(g;) := \psi'(g;)$, for our generators $g \in I'$.
 For this to define an element $[\widetilde{\psi'}]$ in cohomology, we
 need to show
 $\widetilde{\psi'} \circ d_2^{L'} = 0$.
 Observe that $\cF'_2$
 has the form
 \[
 \cF'_2 = \cF_2 \oplus \bigoplus_{i=1}^{r-1} S(s_i;x) \oplus S(s_i;y)
 \oplus S(s_i;z) \oplus S(s_i;w),
 \]
 where $(s_i;x)$ maps to a choice of minimal syzygy arising from the
 fact that $xs_i \in I^{(i-1)}$, and similarly for $(s_i;y)$,
 $(s_i;z)$, and $(s_i;w)$ (recall that $0 \neq s_i +I^{(i-1)} \in \Soc
 A^{(i-1)}$). All generators of the form $(s_i;x)$ or $(s_i;y)$ have
 bidegree $(n_1, n_2-1)$, while the generators of the form $(s_i;z)$
 or $(s_i;w)$ have bidegree $(n_1-1, n_2)$.  Since $j<-1$, we see
 $\bideg(\psi') \notin \{ (-1,0), (0,-1) \}$.  Then identifying
 (bi)graded pieces of $B$ and $A$---like in the proof of Lemma
 \ref{lem:LES-T1}---shows that $\widetilde{\psi'} \circ d_2^{L'} = 0$,
 and therefore that $[\psi']=0$.

 It remains to show that $[\psi']$ is trivial when $\bideg(\psi') \in
 \{(-1,0), (0,-1)\}$.  We show this directly rather than lifting to
 $\widetilde{\psi'}$.  Suppose $\bideg(\psi') = (-1,0)$.  Note that
 $\psi' \colon F'/I'F' \to B'$ is an $A'$-linear map satisfying $\psi'
 \circ d_2^{L'} = 0$ and so factors through $I'/I'^2$; for simplicity,
 we work with the corresponding $S$-linear map $\varphi' \colon I' \to
 B'$.  Considering $\bideg(\psi')$, the generators of $I \subset I'$
 have values
 \[
 \varphi'(x^{n_1-k}y^k) = \sum_{0\le i<n_1}
 y^ix^{n_1-1-i}a^{(k)}_{i,0}, \quad \varphi'(z^{n_2-\ell}w^{\ell}) =
 0, \quad \text{ and } \quad \varphi'(q) = az - a'w,
 \]
 where $a,a',a^{(k)}_{i,0} \in \kk$, for all $0 \le k \le n_1$.
 Observe that relation \eqref{syz:qxy} holds for $\varphi'$ and takes
 place in $B'_{(n_1-1,1)} = A'_{(n_1-1,1)} = A_{(n_1-1,1)}$.  This
 implies that the proof of Proposition \ref{prop:tnt1} applies
 verbatim to these values of $\varphi'$; in other words, $\varphi'$
 acts as a derivative map on the generators of $I$, and so
 $\varphi'|_{I} = \delta|_{I}$, where $\delta := \pi' \circ
 (a\partial_x + a'\partial_y)$ and $a\partial_x + a'\partial_y \colon
 I' \to A'$.
 As $xs_1, ys_1 \in I$, we have $x\varphi'(s_1) = x\delta(s_1)$ and
 $y\varphi'(s_1) = y\delta(s_1)$, and we see that $\varphi'(s_1) -
 \delta(s_1) \in (\Soc B')_{(n_1-2, n_2-1)} = 0$, which means
 $\varphi'(s_1) = \delta(s_1)$.  Thus, $\varphi'|_{I^{(1)}} =
 \delta|_{I^{(1)}}$ holds.  Repeating this for $xs_2, ys_2 \in
 I^{(1)}$ now shows that $\varphi'|_{I^{(2)}} = \delta|_{I^{(2)}}$
 holds, and continuing we eventually obtain $\varphi' = \delta$.  In
 other words, $\psi'$ is a trivial negative tangent vector,
 and a symmetric argument applies to the case $\bideg(\psi') =
 (0,-1)$.  Therefore, we have $T^1(A'/\kk,B')_{<0} = 0$.

 The argument to show $T^1(B'/A',B')_{<0} = 0$ mirrors the proof of
 Proposition \ref{prop:tnt2}, as $\pi' \colon A' \to B'$ is a
 surjection and we are assuming that $\Soc B' = B'_{(n_1-1, n_2-1)}$.
 Hence, the long exact sequence proves that $T^1(B'/\kk,B')_{<0} = 0$.
\end{proof}

\section{Vanishing Nonnegative Obstruction Spaces, II}
\label{sec:obs2}

Continuing with the notation from \S\ref{sec:tnt2},
our goal here is to prove that $T^2(B/\kk,B)_{\ge 0} = 0$.  As before,
the pair of natural ring maps $\kk \to A \to B$ yields a long exact
sequence, which terminates as follows:
\begin{equation}
\label{eqn:T2les}
\dotsb \longrightarrow T^2(B/A,B) \longrightarrow T^2(B/\kk,B)
\longrightarrow T^2(A/\kk,B).
\end{equation}
Let us first examine $T^2(B/A,B)$.  The truncated cotangent complex of
$\pi \colon A \to B$ is described in the proof of Lemma
\ref{lem:LES-T1}; it is
\[
L_{B/A,\bullet} \colon 0 \longleftarrow B(-n_1-n_2+2) \longleftarrow
\fmm_A(-n_1-n_2+2),
\]
where the differential is a twist of $\fmm_A \subset A \to B$.  By
definition, this implies $T^2(B/A,B)$ is a quotient of
$\Hom_{B}(\fmm_A(-n_1-n_2+2), B)$, the latter being trivial in
nonnegative degrees---i.e.\ $\fmm_A(-n_1-n_2+2)$ is generated in
degree $n_1+n_2-1$ while $B_i = 0$, for all $ i > n_1+n_2-2$.  This
shows that
\begin{equation}\label{eqn:T2BoverAgeq0}
T^2(B/A,B)_{\ge 0} = 0.
\end{equation}

Recall that the truncated cotangent complex
of $\kk \to A$
is described in \eqref{cotcplxAk} and equals
\[
L_{\bullet} \colon A(-1)^4 \longleftarrow \cF_{1} \otimes_S A
\longleftarrow \cF_{2}/\Kos'.
\]
We show the following.

\begin{proposition}
  \label{prop:obs2}
  Let $B := S/J$, where $J$ is as in \S\ref{sec:tnt2}.  We have
  $T^2(B/\kk,B)_{\ge 0} = 0$.
\end{proposition}

\begin{proof} 
  We must examine $T^2(A/\kk,B)$.  Suppose we are given an $A$-linear
  homomorphism $\psi \colon \cF_{2}/\Kos' \to B$ of nonnegative
  degree; decomposing $\psi$, we may assume that $\psi$ has
  $\bideg(\psi) = (j_1,j_2) \in \ZZ^2$ such that $j_1+j_2 \ge 0$. We
  may think of $\psi$ as an $S$-linear map $\cF_{2} \to B$ vanishing
  on $\Kos'$.  As $\cF_{2}$ is free, there is a bigraded lifting
  $\widetilde{\psi} \colon \cF_{2} \to A$ such that $\pi \circ
  \widetilde\psi = \psi$.  If $\widetilde\psi|_{\Kos'} = 0$, then
  $\widetilde\psi$ defines an element of $T^2(A/\kk,A)_{\ge 0} = 0$,
  so $\widetilde\psi$ factors through $d_2^{L}$; then $\psi$ also
  factors through $d_2^L$, showing that $\psi$ is trivial in
  $T^2(A/\kk,B)$.

  So it remains to show that $\widetilde\psi|_{\Kos'} = 0$. Let $G \in
  \Kos'$ be a minimal generator and set $(i_1,i_2) := \bideg(G)$ so
  that
  \[
  (i_1,i_2) \in \{ (2n_1,0), (0,2n_2), (n_1,n_2), (n_1+1,1), (1,n_2+1)
  \}.
  \]
  That is, the Koszul relation $x^{n_1-i}y^i (x^{n_1-i'}y^{i'};) -
  x^{n_1-i'}y^{i'} (x^{n_1-i}y^{i};)$ has bidegree $(2n_1,0)$; the
  Koszul relation $x^{n_1-i}y^i (q;) - q (x^{n_1-i}y^{i};)$ has
  bidegree $(n_1+1,1)$; etc.  Without loss of generality, assume $n_1
  \le n_2$. Since $\pi \circ \widetilde\psi = \psi$, we have
  \[
  \widetilde\psi(G) \in J/I \subset A_{(n_1-1,n_2-1)}.
  \]
  As $\widetilde\psi(G) \in A_{(i_1+j_1,i_2+j_2)}$ and $j_1+j_2 \ge
  0$, we immediately find that
  \[
  (j_1,j_2) \notin \{ (-n_1-1,n_2-1), (-2,n_2-2),
  (n_1-2,-2)\}\quad\Longrightarrow\quad \widetilde\psi(G)=0.
  \]
  In particular, for such $(j_1,j_2)$, we have
  $T^2(A/\kk,B)_{(j_1,j_2)} = 0$; combining this with
  \eqref{eqn:T2les} and \eqref{eqn:T2BoverAgeq0}, we find
  $T^2(B/\kk,B)_{(j_1,j_2)} = 0$.

  To complete the proof, we must show that $T^2(B/\kk,B)_{(j_1,j_2)} =
  0$ for $(j_1,j_2)$ belonging to $\{ (-n_1-1,n_2-1), (-2,n_2-2),
  (n_1-2,-2)\}$. Consider the long exact sequence induced by the ring
  maps $A \to A/\fmm_A=\kk \to B$, which contains the following
  portion:
  \[
  \dotsb \longrightarrow T^1(\kk/A,B) \longrightarrow T^2(B/\kk,B)
  \longrightarrow T^2(B/A,B) \longrightarrow \dotsb.
  \]
  Because $A \to \kk$ is surjective,
  we have $T^1(\kk/A, B)_{(j_1,j_2)} = \Hom_{\kk}(\fmm_A/\fmm_A^2,
  B)_{(j_1,j_2)}=0$ since $j_1$ or $j_2$ is at most $-2$;
  cf.\ \cite[Proposition 3.8]{Hartshorne--2010}.  Since
  \[
  T^2(B/A,B)_{(j_1,j_2)} = 0
  \]
  holds by \eqref{eqn:T2BoverAgeq0}, restricting this long exact
  sequence to the bidegree $(j_1,j_2)$ part shows that
  $T^2(B/\kk,B)_{(j_1,j_2)} = 0$.
  Hence, we have $T^2(B/\kk,B)_{\ge 0} = 0$.
\end{proof}

Again, the argument iterates.

\begin{corollary}
  \label{cor:obs2}

  Let $s_1, s_2, \dotsc, s_{r} \in S$,
  \[
  J' = I + \llrr{s_1, s_2, \dotsc, s_{r}}
  \]
  and $B' = S/J'$. Assume every $s_i+I\in \Soc A$ and $\Soc B' =
  B'_{(n_1-1,n_2-1)}$. Then $J'$ has vanishing nonnegative obstruction
  space.
\end{corollary}

\begin{proof}
  We prove the result by induction on $r$.  Proposition
  \ref{prop:obs2} handles the case $r = 1$, so we take $r > 1$. Let
  $I^{(0)} := I$, $I^{(i)} := I + \llrr{s_1, s_2, \dotsc, s_i}$, and
  $A^{(i)}=S/I^{(i)}$ for $1 \le i \le r$.  We may further suppose
  that $s_{i+1} + I^{(i)} \in \Soc A^{(i)}$ is nonzero. Note that by
  Lemma \ref{l:socle-r-condition->socle-i-condition}, $\Soc A^{(i)} =
  A^{(i)}_{(n_1-1, n_2-1)}$ for all $1 \leq i \leq r$.
 
  Set $I' = I^{(r-1)}$, $A' = S/I'$, and $\pi' \colon A' \to B'$. 
  We use the long exact sequence of the pair of ring maps $\kk \to A'
  \to B'$, studying the portion
  \[
  \dotsb \longrightarrow T^2(B'/A',B') \longrightarrow T^2(B'/\kk,B')
  \longrightarrow T^2(A'/\kk,B').
  \]
  Our assumptions guarantee that the proof of the equality
  $T^2(B'/A',B')_{\ge 0} = 0$ follows exactly as
  in the case $r = 1$.

  As mentioned in the proof of Corollary \ref{cor:tnt2}, the minimal
  free resolution of $A'$ over $S$ has terms
  \[
  \cF'_{\bullet} \colon S \longleftarrow \cF_1 \oplus \bigoplus_{i =
    1}^{r-1} S(s_i;) \longleftarrow \cF_2 \oplus \bigoplus_{i=1}^{r-1}
  \bigl( S(s_i;x) \oplus S(s_i;y) \oplus S(s_i;z) \oplus S(s_i;w)
  \bigr) \longleftarrow \dotsb.
  \]
  We wish to apply the proof of Proposition \ref{prop:obs2} to an
  $A'$-linear map $\psi' \colon \cF'_2 / \Kos'' \to B'$.  Note that
  the generators of $\cF'_2$ not belonging to $\cF_2$ all have degree
  $n_1+n_2-1$.  This implies that any $\psi'$ of nonnegative degree
  must vanish on these generators and any syzygies involving them.
  The same analysis of bidegrees as in the proof of Proposition
  \ref{prop:obs2} then holds, showing that $T^2(A'/\kk,B')_{\ge 0} =
  0$.
  Hence, we have $T^2(B'/\kk,B')_{\ge 0} = 0$ by the long exact
  sequence.
\end{proof}

\section{Dimensions of Components}
\label{sec:dimOfComps}

Let $I$ be as in Theorem \ref{thm:intro}.  Having now shown that $[I]$
is a smooth point of the Hilbert scheme and that the irreducible
component containing $[I]$ is elementary, we compute the dimension of
this component; see Corollary \ref{cor:verySmall}.  This is achieved
by explicitly computing the dimension of the tangent space
$\Hom_S(I,S/I)$.

Let $\varphi\in\Hom_S(I,S/I)$.  Our starting point is to re-examine
relation (\ref{syz:qxy}), namely
\[
x^{n_1-1-k}y^k \varphi(xz-yw) = z \varphi(x^{n_1-k}y^k) - w
\varphi(x^{n_1-1-k}y^{k+1}), 
\]
where $0 \le k \le n_1-1$.  

\begin{proposition}
  \label{prop:grid}
  Let $q:=xz-yw$, $r := \varphi(q)$, and $p_k :=
  \varphi(x^{n_1-k}y^k)$, for $0 \le k \le n_1$.  Let $p'_k \in
  \Ann_{x,y}$ be such that $p_k-p'_k$ is supported away from
  $\Ann_{x,y}$.  For any $f\in S/I$,
  let $f_{i,j}$ be as in Lemma \ref{l:decomposition-of-polynomials}.
  Each $(p'_k)_{i,j}$ factors as $x^{n_1-1-i} (p'_k)_{i,j}^z$, where
  $(p'_k)_{i,j}^z$ is some polynomial in $z$.  Let $r_{0,j}^z$ denote
  the $x^0z^{\ge 0}$-part of $r_{0,j}$.
  
  Then relation (\ref{syz:qxy}) is equivalent to the equations
  \[ \tag{$\star^{z}$} \label{starz}
  \begin{cases}
    z(p_{k}')_{i,0}^z = (n_1-k)z^{k-i} r_{0,k-i}^z, &\text{ for } k
    \le i < n_1, \\
    z(p_{k+1}')_{i+1,0}^z = z^{k-i} \left( z (p_{0}')_{0,k-i}^z -
    (k+1)r_{0,k-i}^z \right), &\text{ for } 0 \le i < k, \\
    z^j(p_{k+1}')_{0,j-1}^z = z^{k+j} \left( z(p_0')_{0,k+j}^z -
    (k+1) r_{0,k+j}^z \right), &\text{ for } 0 < j < n_2.
  \end{cases}
  \]
\end{proposition}

Note that terms in (\ref{starz}) may vanish per Lemma
\ref{l:decomposition-of-polynomials}, e.g.~if $k<i$, then $r_{0,k-i}^z
= 0$.

\begin{proof}
  In this notation, (\ref{syz:qxy}) becomes
  \[
  x^{n_1-1-k}y^k r = z p_k - w p_{k+1},
  \]
  where $0\le k < n_1$.  Applying Lemma
  \ref{l:decomposition-of-polynomials} to the left-hand side, we have
  \begin{align*}
    x^{n_1-1-k}y^kr &= x^{n_1-1-k}y^k \left( \sum_{0<i<n_1} y^i
    r_{i,0} + \sum_{0<j<n_2} w^j r_{0,j} + r_{0,0}\right) \\ &=
    \sum_{0<j<n_2} x^{n_1-1-k}y^k w^j r_{0,j}^z +
    x^{n_1-1-k}y^kr_{0,0}^z,
  \end{align*}
  where $r_{0,j}^z$ denotes the $x^0z^{\ge 0}$-part of the polynomial
  $r_{0,j} = r_{0,j}(x,z)$, for $j \ge 0$.  This equals
  \begin{align*}
    &= \sum_{0<j\le k} y^{k-j} x^{n_1-1-k+j} z^j r_{0,j}^z +
    \sum_{k<j<n_2} w^{j-k} x^{n_1-1} z^k r_{0,j}^z + x^{n_1-1-k} y^k
    r_{0,0}^z \\ &= \sum_{0 \le i < k} y^i x^{n_1-1-i} z^{k-i}
    r_{0,k-i}^z + y^k x^{n_1-1-k} r_{0,0}^z + \sum_{0<j<n_2-k} w^j
    x^{n_1-1} z^k r_{0,k+j}^z \\ &= \sum_{0<i\le k} y^i x^{n_1-1-i}
    z^{k-i} r_{0,k-i}^z + \sum_{0<j<n_2-k} w^j x^{n_1-1} z^k
    r_{0,k+j}^z + x^{n_1-1} z^k r_{0,k}^z,
  \end{align*}
  which is the expression guaranteed by Lemma
  \ref{l:decomposition-of-polynomials} for the element $x^{n_1-1-k}y^k
  r \in A = S/I$.

  Now consider the right-hand side $zp_k - wp_{k+1}$.
  A straightforward calculation with the expressions from
  Corollary \ref{cor:mindingPsAndQsInGeneral}  shows that $z p_k - w p_{k+1} = z
  p_k' - w p_{k+1}'$,
  where $p_k'$ is the part of $p_k$ annihilated by $x$ and $y$.  For
  any element $f = \sum_{i>0} y^i f_{i,0} + \sum_{j>0} w^j f_{0,j} +
  f_{0,0}$ expressed using Lemma \ref{l:decomposition-of-polynomials}
  and belonging to $\Ann_{x,y} = A_{(n_1-1,*)}$, we may assume
  $f_{i,j}$ has the form $x^{n_1-1-i} f_{i,j}^z$, where $f_{i,j}^z$ is
  a polynomial in $z$.  Thus, we have
  \begin{align*}
    zp_k' &= z\left( \sum_{0<i<n_1} y^i (p_k')_{i,0} + \sum_{0<j<n_2}
    w^j (p_k')_{0,j} + (p_k')_{0,0} \right) \\ &= z\left(
    \sum_{0<i<n_1} y^i x^{n_1-1-i} (p_k')_{i,0}^z + \sum_{0<j<n_2} w^j
    x^{n_1-1} (p_k')_{0,j}^z + x^{n_1-1} (p_k')_{0,0}^z \right) \\ &=
    \sum_{0<i<n_1} y^i x^{n_1-1-i} z (p_k')_{i,0}^z + \sum_{0<j<n_2-1}
    w^j x^{n_1-1} z (p_k')_{0,j}^z + x^{n_1-1} z (p_k')_{0,0}^z.
  \end{align*}
  We also have
  \begin{align*}
    wp_{k+1}' &= w\left( \sum_{0<i<n_1} y^i (p_{k+1}')_{i,0} +
    \sum_{0<j<n_2} w^j (p_{k+1}')_{0,j} + (p_{k+1}')_{0,0} \right)
    \\ &= w\left( \sum_{0<i<n_1} y^i x^{n_1-1-i} (p_{k+1}')_{i,0}^z +
    \sum_{0<j<n_2} w^j x^{n_1-1} (p_{k+1}')_{0,j}^z + x^{n_1-1}
    (p_{k+1}')_{0,0}^z \right) \\ &= \sum_{0<i<n_1} y^{i-1} x^{n_1-i}
    z (p_{k+1}')_{i,0}^z + \sum_{0<j<n_2-1} w^{j+1} x^{n_1-1}
    (p_{k+1}')_{0,j}^z + w x^{n_1-1} (p_{k+1}')_{0,0}^z \\ &=
    \sum_{0<i<n_1-1} y^{i} x^{n_1-1-i} z (p_{k+1}')_{i+1,0}^z +
    \sum_{0<j<n_2} w^{j} x^{n_1-1} (p_{k+1}')_{0,j-1}^z + x^{n_1-1} z
    (p_{k+1}')_{1,0}^z
  \end{align*}
  and combining these gives
  \begin{align*}
    zp_k' - wp_{k+1}' &= \sum_{0<i<n_1} y^i x^{n_1-1-i} z \left(
    (p_k')_{i,0}^z - (p_{k+1}')_{i+1,0}^z \right) \\ &\qquad\qquad +
    \sum_{0<j<n_2} w^j x^{n_1-1} \left( z (p_k')_{0,j}^z -
    (p_{k+1}')_{0,j-1}^z \right) + x^{n_1-1} z \left( (p_k')_{0,0}^z -
    (p_{k+1}')_{1,0}^z \right) \!,
  \end{align*}
  where $i = n_1-1$ implies $(p_{k+1}')_{i+1,0}^z$ and $j = n_2-1$
  implies $w^j z = 0$.  Thus, (\ref{syz:qxy}) is equivalent to the
  conditions
  \[
  \begin{cases}
    z^{k-i} r_{0,k-i}^z = z (p_k')_{i,0}^z - z (p_{k+1}')_{i+1,0}^z,
    &\text{ for } 0 \le i < n_1 \\ w^j z^k r_{0,k+j}^z = w^j z
    (p_k')_{0,j}^z - w^j (p_{k+1}')_{0,j-1}^z, &\text{ for } 0 < j <
    n_2,
  \end{cases}
  \]
  where terms may vanish for certain values of their indices (as
  indicated in the preceding summation notation).  Moreover, these
  conditions are unchanged when $w$ is replaced by $z$.  When $i \ge
  k$, we rewrite the first condition as
  \begin{align*}
    z (p_{k}')_{i,0}^z &= z (p_{k+1}')_{i+1,0}^z + z^{k-i} r_{0,k-i}^z
    = z (p_{k+2}')_{i+2,0}^z + 2z^{k-i} r_{0,k-i}^z = \dotsb \\ &= z
    (p_{n_1}')_{i+n_1-k,0}^z + (n_1-k)z^{k-i} r_{0,k-i}^z =
    (n_1-k)z^{k-i} r_{0,k-i}^z,
  \end{align*}
  as $(p_{n_1}')_{n_1+i-k,0}^z = 0$ by definition.  When $0 \le i <
  k$, we rewrite the first condition differently as
  \begin{align*}
    z (p_{k+1}')_{i+1,0}^z &= z (p_{k}')_{i,0}^z - z^{k-i} r_{0,k-i}^z
    = z (p_{k-1}')_{i-1,0}^z - 2z^{k-i} r_{0,k-i}^z = \dotsb \\ &= z
    (p_{k-i}')_{0,0}^z - (i+1)z^{k-i} r_{0,k-i}^z \\ &= z^{k-i} \left(
    z (p_0')_{0,k-i}^z - (k-i)r_{0,k-i}^z \right) - (i+1)z^{k-i}
    r_{0,k-i}^z \\ &= z^{k-i} \left( z(p_0')_{0,k-i}^z -
    (k+1)r_{0,k-i}^z \right),
  \end{align*}
  assuming the third condition in the statement of the proposition
  holds.  To obtain the latter, we rewrite the second condition above
  as
  \begin{align*}
    z^j (p_{k+1}')_{0,j-1}^z &= z^{j+1} (p_k')_{0,j}^z - z^{k+j}
    r_{0,k+j}^z = z^{j+2} (p_{k-1}')_{0,j+1}^z - 2z^{k+j} r_{0,k+j}^z
    = \dotsb \\ &= z^{j+k+1} (p_{0}')_{0,j+k}^z - (k+1)z^{k+j}
    r_{0,k+j}^z.
  \end{align*}
  Hence, we obtain the desired conditions.
\end{proof}

\begin{corollary}
  The unique irreducible component of $\hilb^{d}(\bbA^4)$ containing
  $[I]$ has dimension
  \[
  D = D(n_1,n_2) := F(n_1,n_2) + F(n_2,n_1) +
  d(n_1,n_2) -1,
  \]
  where
  \[
  F(a,b) := \sum_{i=2}^{a-1} (i-1)\binom{b-i}{1} +
  (a-1)\binom{b-a+1}{2} + (a+1)(a+b-1) + \binom{b-1}{2},
  \]
  $d = d(n_1,n_2) = \frac{n_1n_2}{2}(n_1+n_2)$, and $\binom{j}{k}$
  denotes $\frac{j!}{k!(j-k)!}$ if $j \ge k \ge 0$ and is $0$
  otherwise.
\end{corollary}

\begin{proof}
  We shall compute the dimension of the tangent space $\Hom_S(I, A)$
  at the smooth point $[I]$ rather directly.  A homomorphism $\varphi
  \colon I \to A$ is determined by its values
  \[
  r = \varphi(xz-yw), p_k = \varphi(x^{n_1-k}y^k), q_{\ell} =
  \varphi(z^{n_2-\ell}w^{\ell}), \text{ for } 0 \le k \le n_1 \text{
    and } 0 \le \ell \le n_2.
  \]
  There are four kinds of relations that put restrictions on
  coefficients, namely:
  \begin{enumerate}
  \item relations among $p_0, p_1, \dotsc, p_{n_1}$ described in
    Corollary \ref{cor:mindingPsAndQsInGeneral};
  \item relations among $q_0, q_1, \dotsc, q_{n_2}$ described in
    Corollary \ref{cor:mindingPsAndQsInGeneral}  with $n_1$ and $n_2$ swapped;
  \item relations among $p_0, p_1, \dotsc, p_{n_1}$ and $r$ described
    in Proposition \ref{prop:grid};
  \item relations among $q_0, q_1, \dotsc, q_{n_2}$ and $r$ described
    in Proposition \ref{prop:grid} with $n_1$ and $n_2$ swapped.
  \end{enumerate}
  Moreover, these conditions are independent, in the sense that (i)
  only restricts the coefficients
  of the $p_k-p'_k$, whereas (iii) only restricts the coefficients of
  the $p_k'$'s and uses the coefficients of $r$ as parameters
  (ensuring (iii) and (iv) are independent).
  
  Starting with (i), we apply Corollary
  \ref{cor:mindingPsAndQsInGeneral} to the sequence $p_0, p_1, \dotsc,
  p_{n_1}$ to show
  \begin{align*}
    p_k &= p_k' + \sum_{i=0}^{k} x^{n_1-k}y^{k-i}z^it_i +
    \sum_{i=k+1}^{n_1} x^{n_1-i}z^kw^{i-k}t_i \\ &= p_k' +
    \sum_{i=2}^{k} x^{n_1-k}y^{k-i}z^it_i + \sum_{i=k+1}^{n_1}
    x^{n_1-i}z^kw^{i-k}t_i,
  \end{align*}
  where $p_k'$ is the part of $p_k$ annihilated by $x$ (equivalently
  $y$); the second equality follows as $i=0$ implies $x^{n_1-k}y^{k-0}
  = 0$ and $i=1$ gives either the term $x^{n_1-1}z^0w^{1-0}t_1 \in
  \Ann(x)$ for $k=0$ or the term $x^{n_1-k}y^{k-1}z^1t_1 \in \Ann(x)$
  for $k > 0$.  Observe that the term containing $t_i$ also has a
  monomial of bidegree $(n_1-i,i)$.  Corollary
  \ref{cor:mindingPsAndQsInGeneral} states that $t_i$ is a polynomial
  in $x,z$ for $2 \le i < n_1$, so the $x$-degree satisfies
  $\deg_{x}(t_i) < i-1$ (a term of $x$-degree $i-1$ would produce an
  element of $\Ann(x)$) and the $z$-degree satisfies $\deg_z(t_i) <
  n_2-i$.  These imply that $t_i$ has $(i-1)(n_2-i)$ free
  coefficients, if $i < n_2$, and $0$ free coefficients, if $i \ge
  n_2$; we denote this number by $(i-1)\binom{n_2-i}{1}$.  Similarly,
  $t_{n_1}$ is a polynomial in $x,z,w$ and $\deg_x(t_{n_1}) < n_1-1$
  and $\deg_{z,w}(t_{n_1}) < n_2-n_1$.  This gives $(n_1-1)
  \binom{2+n_2-n_1-1}{2} = (n_1-1) \binom{n_2-n_1+1}{2}$ free
  coefficients.  Thus, a formula for the contribution of $t_2, \dotsc,
  t_{n_1}$ to the dimension is
  \[  \tag{$t_i$-count} \label{ticount}
  \sum_{i=2}^{n_1-1} (i-1)\binom{n_2-i}{1} +
  (n_1-1)\binom{n_2-n_1+1}{2}.
  \]
  The analogous count using (ii) is obtained by swapping $n_1$ and
  $n_2$.

  For (iii), we use Proposition \ref{prop:grid} to find free
  parameters in each $p_k'$.  The first condition of (\ref{starz})
  says that the coefficients $f_0, f_1, \dotsc, f_{n_2-2}$ of $f_0 +
  f_1z + \dotsb + f_{n_2-1}z^{n_2-1} := (p_k')_{i,0}^z$ are determined
  by those of $r_{0,k-i}^z$; only the coefficient $f_{n_2-1}$ is free,
  so $(p_k')_{i,0}^z$ contributes exactly one degree of freedom.
  (When $k < i$, we have $r_{0,k-i}^z := 0$.  When $k = i$, we get
  $z(p_k')_{k,0}^z = (n_1-k)r_{0,0}^z$, further implying $r_{0,0}^z$
  (and in turn $r$) has trivial constant term---this is expected as
  $I$ has trivial negative tangents and $\deg q = 2$.)  The second
  condition of (\ref{starz}) says that the coefficients $f_0, f_1,
  \dotsc, f_{n_2-2}$ of $f_0 + f_1z + \dotsb + f_{n_2-1}z^{n_2-1} :=
  (p_{k+1}')_{i+1,0}^z$ are determined by those of $(p_0')_{0,k-i}^z$
  and $r_{0,k-i}^z$; the coefficient $f_{n_2-1}$ is free.  The third
  condition of (\ref{starz}) says that the coefficients $f_0, f_1,
  \dotsc, f_{n_2-j-1}$ of $f_0 + f_1z + \dotsb + f_{n_2-j}z^{n_2-j} :=
  (p_{k+1}')_{0,j-1}^z$ are determined by those of $(p_0')_{0,k+j}^z$
  and $r_{0,k+j}^z$; the coefficient $f_{n_2-j}$ is free.  In other
  words, letting $0 < \kappa \le n_1$, every term in the following
  expansion contributes a single degree of freedom:
  \begin{align*}
    p_{\kappa}' &= \sum_{0<\iota<n_1}
    y^{\iota}x^{n_1-1-\iota}(p_{\kappa}')_{\iota,0}^z +
    x^{n_1-1}(p_{\kappa}')_{0,0}^z + \sum_{0<\eta<n_2}
    w^{\eta}x^{n_1-1}(p_{\kappa}')_{0,\eta}^z, \\ &=
    \sum_{0<\iota<\kappa}
    y^{\iota}x^{n_1-1-\iota}(p_{\kappa}')_{\iota,0}^z +
    \sum_{\kappa\le \iota<n_1}
    y^{\iota}x^{n_1-1-\iota}(p_{\kappa}')_{\iota,0}^z
    \\ &\qquad\qquad\qquad\qquad + x^{n_1-1}(p_{\kappa}')_{0,0}^z +
    \sum_{0<\eta<n_2} w^{\eta}x^{n_1-1}(p_{\kappa}')_{0,\eta}^z, \\ &=
    \sum_{0\le i<k} y^{i+1}x^{n_1-2-i}(p_{k+1}')_{i+1,0}^z +
    \sum_{\kappa\le \iota<n_1}
    y^{\iota}x^{n_1-1-\iota}(p_{\kappa}')_{\iota,0}^z
    \\ &\qquad\qquad\qquad\qquad + x^{n_1-1}(p_{k+1}')_{0,0}^z +
    \sum_{1<j<n_2+1} w^{j-1}x^{n_1-1}(p_{k+1}')_{0,j-1}^z,
  \end{align*}
  where $k = \kappa-1$, $i = \iota-1$, and $j = \eta +1$; the only
  term here not covered by (\ref{starz}) is $(p_{k+1}')_{0,n_2-1}^z$,
  which is therefore a free constant.  As each term contributes one
  degree of freedom, the contribution by $p_1', p_2', \dotsc,
  p_{n_1}'$ equals $n_1(n_1+n_2-1)$.  Now considering
  \begin{align*}
    p_{0}' &= \sum_{0<i<n_1} y^{i}x^{n_1-1-i}(p_{0}')_{i,0}^z +
    x^{n_1-1}(p_{0}')_{0,0}^z + \sum_{0<j<n_2}
    w^{j}x^{n_1-1}(p_{0}')_{0,j}^z,
  \end{align*}
  we find that each $(p_{0}')_{i,0}^z$ contributes one degree of
  freedom, for $i \ge 0$, giving $n_1$.  Each $(p_{0}')_{0,j}^z$ is
  free and contributes $n_2-j$ degrees of freedom, for $0 < j < n_2$,
  giving $\sum_{j=1}^{n_2-1} n_2-j = \binom{n_2}{2}$.  Thus, a formula
  for the contribution to the dimension by $p_0', p_1', \dotsc,
  p_{n_1}'$ is
  \[ \tag{$p_k'$-count} \label{pk'count}
  n_1(n_1 + n_2) + \binom{n_2}{2}.
  \]
  The analogous count for (iv) is obtained by swapping $n_1$ and
  $n_2$, and neither (iii) nor (iv) restricts the coefficients of $r$.

  The total contribution by $p_0, p_1, \dotsc, p_{n_1}$ to the
  dimension is therefore the sum of (\ref{ticount}) and
  (\ref{pk'count}), which equals
  \[
  F(n_1,n_2) := \sum_{i=2}^{n_1-1} (i-1)\binom{n_2-i}{1} +
  (n_1-1)\binom{n_2-n_1+1}{2} + n_1(n_1+n_2) + \binom{n_2}{2}.
  \]
  Symmetrically, the contribution by $q_0, q_1, \dotsc, q_{n_2}$ is
  $F(n_2,n_1)$.  To finish, we need the contribution by $r$.  The only
  condition on $r$, imposed by (\ref{starz}) when $k=i$, is of having
  a trivial constant term.  Thus, $r$ contributes $\dim_\kk(S/I)-1$
  dimensions.
  
  Hence, combining with the dimension formula in Lemma
  \ref{l:decomposition-of-polynomials}, we see
  \[
  D = \dim_{\kk} \Hom_S(I,A) = F(n_1,n_2) + F(n_2,n_1) +
  \frac{n_1n_2}{2}(n_1+n_2) -1,
  \]
  as desired.
\end{proof}

\begin{corollary}
  \label{cor:verySmall}
  The dimension $D$ simplifies to
  \[
  D = \frac{1}{3}m^3 + mM^2 + m^2 + 2mM + M^2 -
  \frac{1}{3}m -1,
  \]
  where $m := \min\{n_1,n_2\}$ and $M := \max\{n_1,n_2\}$.  In
  particular, the dimension $D$ of the irreducible component of
  $\hilb^{d}(\bbA^4)$ containing the point $[I]$ satisfies $D<4d$, and
  moreover if $(m,M) \notin \{(2,2), (2,3), (2,4)\}$, then $D<3(d-1)$.
\end{corollary}

\begin{proof}
  Let $a=n_1, b=n_2$ and assume $a \le b$ without loss of generality.
  We first simplify the summations in $F(a,b)$ and $F(b,a)$ coming
  from (\ref{ticount}).  For $F(a,b)$ we have
  \begin{align*}
  \sum_{i=2}^{a-1} (i-1)\binom{b-i}{1} &= \sum_{i=2}^{a-1} (i-1)(b-i)
  = \sum_{i=1}^{a-2} i(b-1-i) = (b-1)\sum_{i=1}^{a-2} i -
  \sum_{i=1}^{a-2} i^2 \\ &= (b-1)\binom{a-1}{2} -
  \binom{a-1}{2}\frac{2(a-2)+1}{3} =
  \binom{a-1}{2} \left( b - \frac{2}{3}a \right)
  \end{align*}
  so that
  \[
  F(a,b) = \binom{a-1}{2} \left( b - \frac{2}{3}a \right) +
  (a-1)\binom{b-a+1}{2} + a(a+b) + \binom{b}{2}.
  \]
  For $F(b,a)$ we have
  \begin{align*}
    \sum_{i=2}^{b-1} (i-1)\binom{a-i}{1} &= \sum_{i=2}^{a-1}
    (i-1)(a-i) = \sum_{i=1}^{a-2} i(a-1-i) = (a-1)\sum_{i=1}^{a-2} i -
    \sum_{i=1}^{a-2} i^2 \\ &= (a-1)\binom{a-1}{2} -
    \binom{a-1}{2}\frac{2a-3}{3} = \binom{a-1}{2} \frac{1}{3}a
  \end{align*}
  so that
  \[
  F(b,a) = \binom{a-1}{2} \frac{1}{3}a + 0 + b(a+b) + \binom{a}{2},
  \text{ as } (b-1)\binom{a-b+1}{2} = 0.
  \]
  Therefore, we find that
  \begin{align*}
    D &= F(a,b) + F(b,a) + d(a,b) -1 \\ &= \binom{a-1}{2} \left( b -
    \frac{1}{3}a \right) + (a-1)\binom{b-a+1}{2} + (a+b)^2 +
    \binom{a}{2} + \binom{b}{2} + \frac{ab(a+b)}{2} -1,
  \end{align*}
  which gives the desired expression $\frac{1}{3}a^3 + ab^2 + a^2 +
  2ab + b^2 - \frac{1}{3}a -1$ when expanded.

  Now we examine when $3(d-1) > D$, or rather
  \begin{align*}
    3(d-1) - D &= 3\frac{ab(a+b)}{2} -3 - \left( \frac{1}{3}a^3 +
    ab^2 + a^2 + 2ab + b^2 - \frac{1}{3}a -1 \right) \\
    &= -\frac{1}{3}a^3 +\frac{3}{2}a^2b + \frac{1}{2}ab^2 -a^2 -2ab
    -b^2 +\frac{1}{3}a -2 \\
    &= \left(\frac{1}{2}a-1\right)b^2 + \left(\frac{3}{2}a^2
    -2a\right)b -\frac{1}{3}a^3 -a^2 +\frac{1}{3}a -2 >0.
  \end{align*}
  If $a=2$, then this reduces to $2b - 8 > 0$, so the inequality is
  satisfied if and only if $b > 4$.  Now let $a > 2$.  Then $3(d-1)-D$
  is quadratic in $b$ with roots
  \[
  r_{\pm}(a) = \frac{2a-\frac{3}{2}a^2 \pm \sqrt{\left(\frac{3}{2}a^2
      -2a\right)^2 - 4\left(\frac{1}{2}a-1\right)\left(-\frac{1}{3}a^3
      -a^2 +\frac{1}{3}a -2 \right)}}{a-2}
  \]
  and discriminant simplifying to $\frac{35}{12}a^4 -\frac{16}{3}a^3
  -\frac{2}{3}a^2 +\frac{16}{3}a -8$; the discriminant is positive for
  $a>2$.  If both roots satisfy $r_{\pm}(a) < a$, then $a \le b$
  implies $3(d-1) - D > 0$, as the $b^2$-term in the expression for
  $3(d-1) - D$ has a positive coefficient.  We check
  \begin{align*}
    a > r_{\pm}(a) &\Longleftrightarrow a^2-2a > 2a-\frac{3}{2}a^2 \pm
    \sqrt{\frac{35}{12}a^4 -\frac{16}{3}a^3 -\frac{2}{3}a^2
      +\frac{16}{3}a -8} \\
    &\Longleftrightarrow \left(\frac{5}{2}a^2 -4a\right)^2 >
    \frac{35}{12}a^4 -\frac{16}{3}a^3 -\frac{2}{3}a^2 +\frac{16}{3}a
    -8 \\
    &\Longleftrightarrow \frac{10}{3}a^4 -\frac{44}{3}a^3
    +\frac{50}{3}a^2 -\frac{16}{3}a +8 > 0.
  \end{align*}
  The latter factors as $\frac{2}{3}(a-2)(5a^3 -12a^2 +a -6)$ and is
  positive for $a\ge 3$.  Hence, the point $[I] \in \hilb^d(\bbA^4)$
  lies on a component of dimension $D < 3(d-1)$.

  When $(a,b) \in \{(2,2), (2,3), (2,4)\}$, we immediately find
  $D(a,b) < 4d(a,b)$.  Hence, here the point $[I] \in \hilb^d(\bbA^4)$
  lies on a component of dimension $D < 4d$.
\end{proof}

\begin{remark}
  When $n_1 = n_2 = n$, we obtain $D = \frac{4}{3}n^3 + 4n^2 -
  \frac{1}{3}n -1$ and $4d = 4n^3$.
\end{remark}

\section{Compendium of Elementary Components}

Collecting the results from \S\S\ref{sec:tnt1}--\ref{sec:dimOfComps},
we prove the theorems from the introduction.


\begin{proof}[Proof of Theorems \ref{thm:intro}, and \ref{thm:intro2}]
Let $I$ be as in Theorem
  \ref{thm:intro}. Proposition \ref{prop:Negative-tangents-for-I}
  shows that every irreducible component containing $[I]$ is
  elementary. Proposition \ref{prop:IisSmooth} then shows that $[I]$
  is a smooth point, so it must lie on a unique irreducible
  component. The formula for $d=\dim_\kk S/I$ is given in Lemma
  \ref{l:decomposition-of-polynomials}. Lastly, the formula for the
  dimension $D$ of the irreducible component containing $[I]$ is given
  in Corollary \ref{cor:verySmall}, where it is also shown that $D<4d$
  and if $(m,M)\notin\{(2,2),(2,3),(2,4)\}$, then $D<3(d-1)$. This
  proves Theorem \ref{thm:intro}.

  Finally, let $J$ be as in Theorem \ref{thm:intro2}. Corollary
  \ref{cor:tnt2} shows that $J$ has trivial negative tangents, so
  every component containing it is elementary by \cite[Theorem
    1.2]{Jelisiejew--2019}. Corollary \ref{cor:obs2} then proves that
  $[J]$ is a smooth point, so it must lie on a unique irreducible
  component.  Lastly, Lemma \ref{lem:socB} shows that the condition on
  the socle is automatic when $r=1$.  This proves Theorem
  \ref{thm:intro2}.
\end{proof}

\begin{proof}[Proof of Theorem \ref{thm:main-result}]
  Let $I$, $d$, and $D$ be as in Theorem \ref{thm:intro}. By Theorem \ref{thm:Jel}, there is an open subset $U\subset X:=\hilb^+_{d}(\bbA^4) \times \bbA^4$ such that $\theta|_U\colon U\to \hilb^d(\bbA^4)$ is an open immersion with $\theta|_U\bigl(([I],0)\bigr)=[I]$. Consider the cartesian diagram
  \[
  \xymatrix{
  U\ar[r] & X\ar[r]^-{\theta} & \hilb^d(\bbA^4)\\
  V\ar[r]\ar@{^{(}->}[u] & Z\ar[r]\ar@{^{(}->}[u] & \mathcal{Z}\ar@{^{(}->}[u]
  }
  \]
  where $\mathcal{Z}$ is the locus of points $[J]$ supported at the origin. By Theorem \ref{thm:intro}, we know $[I]$ is a smooth point contained on a unique elementary component, so shrinking $U$ if necessary, we may assume $U$ is smooth, irreducible, and of dimension $D$.
  
  We show that every irreducible component of $V$ containing $([I],0)$ has dimension at most $D-4$. For the purposes of computing dimensions, it suffices to replace $V$ by its reduction $V_{\textrm{red}}$. Since $V\subset Z$ is open, $V_{\textrm{red}}=V\times_Z Z_{\textrm{red}}$ and so we may replace $Z$ by any closed subscheme of $X$ whose reduction agrees with $Z_{\textrm{red}}$. By definition of $\theta$, the $\kk$-points of $Z$ are precisely those of $X_0:=\hilb^+_{d}(\bbA^4) \times \{0\}$. Thus, $X_0$ and $Z$ have the same reduction. We have therefore reduced to showing that $\dim(U\cap X_0)\leq \dim(U)-4$.
  
  Since $U\subset X$ is open and $X=\hilb^+_{d}(\bbA^4)\times\bbA^4 \to \bbA^4$ is flat, we have a flat map $f\colon U\to\bbA^4$. Then $U\cap X_0$ is the fibre of $f$ over $0$. As $U$ and $\bbA^4$ are smooth irreducible $\kk$-schemes, we see $\dim(U\cap X_0)= \dim(U)-4=D-4$.
  
  Hence, this shows that every irreducible component of $\mathcal{Z}$ containing $[I]$ has dimension at most $D-4$. Theorem \ref{thm:intro} tells us $D<3(d-1)$ for $(m,M)\notin\{(2,2),(2,3),(2,4)\}$ and one verifies $D-4 < 3(d-1)$ when $(m,M) \in \{(2,3),(2,4)\}$.
\end{proof}

We end the paper by discussing some examples and variations of
Theorems \ref{thm:intro} and \ref{thm:intro2}.  For a survey of the
known elementary components prior to our work, see \cite[Remark
  6.10]{Jelisiejew--2019} and the section on smoothability in
\cite[Appendix B]{Abdallah--Emsalem--Iarrobino--2021}.

\begin{example}
  \label{eg:1}
  Plugging in some values of $(n_1,n_2)$ to define $I$, we obtain the
  following:
  \begin{itemize}
  \item for $(n_1,n_2) = (2,2)$, the point $[I] \in \hilb^8(\bbA^4)$
    lies on the $25$-dimensional component first discovered by
    Iarrobino--Emsalem \cite[\S 2.2]{Iarrobino--Emsalem--1978};
  \item for $(n_1,n_2) = (2,3)$, the point $[I] \in
    \hilb^{15}(\bbA^4)$ is smooth on a $44$-dimensional component;
  \item for $(n_1,n_2) = (2,4)$, the point $[I] \in \hilb^{24}(\bbA^4)$ is
    smooth on a $69$-dimensional component;
  \item for $(n_1,n_2) = (2,5)$, the point $[I] \in
    \hilb^{35}(\bbA^4)$ is smooth on a $100$-dimensional component;
  \item for $(n_1,n_2) = (3,3)$, the point $[I] \in \hilb^{27}(\bbA^4)$ is
    smooth on a $70$-dimensional component;
  \item for $(n_1,n_2) = (3,4)$, the point $[I] \in
    \hilb^{42}(\bbA^4)$ is smooth on a $104$-dimensional component.
  \end{itemize}
  The scheme $\hilb^{35}(\bbA^4)$ is already known to have a
  $124$-dimensional elementary component, denoted $\mathcal{Z}(3)$ in
  \cite{Jelisiejew--2019}.
  Thus, this Hilbert scheme has at least two elementary components.
\end{example}

\begin{example}
  \label{eg:2}
  Consider the following example where Theorem \ref{thm:intro2} holds.
  Define $I$ using $(n_1, n_2) = (3,3)$ and set
  $s_1 := x^2z^2$, $s_2 := x^2w^2$, and $s_3 := y^2z^2$.  For $0\leq
  i\leq 3$, let $I^{(i)} := I + \llrr{s_1, \dotsc, s_i}$ and
  $A^{(i)}=S/I^{(i)}$.  One can verify that
  \begin{itemize}
    \item the point $[I^{(0)}] \in \hilb^{27}(\bbA^4)$ is smooth on a
      $70$-dimensional component (as above);
    \item the point $[I^{(1)}] \in \hilb^{26}(\bbA^4)$ is smooth on a
      $77$-dimensional component;
    \item the point $[I^{(2)}] \in \hilb^{25}(\bbA^4)$ is smooth on an
      $82$-dimensional component;
    \item the point $[I^{(3)}] \in \hilb^{24}(\bbA^4)$ is smooth on an
      $85$-dimensional component.
  \end{itemize}
  That is, each $[I^{(i)}] \in \hilb^{27-i}(\bbA^4)$ is a smooth point
  on an elementary component of dimension less than that of the main
  component, namely $4(27-i)$.  Furthermore, we see that
  $\hilb^{24}(\bbA^4)$ has at least two elementary components, by
  comparing with Example \ref{eg:1}.
\end{example}

\begin{example}
  \label{eg:2'}
  Further examples similar to Example \ref{eg:2} can be found.  For
  instance, the ideals
  \begin{align*}
    I &= \llrr{x,y}^4 + \llrr{z,w}^4 + \llrr{xz-yw, xy^2w^3, x^3w^3,
      y^3zw^2, y^3z^3} \quad\text{and} \\
    I &= \llrr{x,y}^3 + \llrr{z,w}^5 + \llrr{xz-yw}
  \end{align*}
  lie on distinct elementary components of $\hilb^{60}(\bbA^4)$ of
  respective dimensions $179$ and $146$.
\end{example}

\begin{example}
  The point $[I] \in \hilb^{1000}(\bbA^4)$ defined by
  \[
  I = \llrr{x,y}^{10} + \llrr{z,w}^{10} + \llrr{xz-yw}
  \]
  lies on an elementary component of dimension $1729$.  The
  dimension $1729$ is the second \emph{taxicab number}, i.e.\ the
  minimal positive integer expressible as a sum of two distinct
  cubes in two different ways: $1729 = 9^3+10^3 = 1^3+12^3$.
\end{example}

\begin{example}
  \label{eg:3}
  As mentioned in the introduction (see Question \ref{q:cactus}),
  producing a local, zero-dimensional Gorenstein quotient of $S$ with
  trivial negative tangents gives a way to distinguish cactus and
  secant varieties, see \cite[Proposition
    7.4]{Buczynski--Jelisiejew--2017}. Our techniques allow us to produce 
    an example with socle-dimension $2$ (as opposed to
  socle-dimension $1$).
     
  Let us return to the setting of Example \ref{eg:2}.  Letting $(n_1,
  n_2) = (3,3)$, we find that
  \[
  \Soc A =
  \llrr{x^2z^2 +I, x^2zw +I, x^2w^2 +I, xyz^2 +I, y^2z^2 +I}
  \]
  is $5$-dimensional.
  Then, $$\Soc A^{(3)} = \llrr{x^2zw +I^{(2)}, xyz^2 +I^{(2)}}$$ is
  $2$-dimensional. By Theorem \ref{thm:intro2} and Remark
  \ref{rmk:more-details-intro2}, $A^{(3)}$ has trivial negative
  tangents and vanishing nonnegative obstruction space.
\end{example}

\begin{remark}
  \label{rmk:variants}
  Natural variants of the ideals in Theorem \ref{thm:intro} also
  produce trivial negative tangents.  For instance, Table \ref{table}
  displays some triples $(n_1,n_2,n_3)$ that determine ideals
  \[
  J := \llrr{x,y}^{n_1} + \llrr{z,w}^{n_2} + \llrr{xz-yw, (xz)^{n_3}}
  \]
  with trivial negative tangents, verified by direct computations in
  \emph{Macaulay2} \cite{Grayson--Stillman}.  (For any ideal, it
  suffices to check that $T^1_{A,i} = 0$ for finitely many $i<0$;
  e.g.\ if $I$ is homogeneous, then $\Hom(I,S/I)_{<-N} = 0$, where $N$
  is the highest degree of a generator of $I$.)  Letting $B := S/J$,
  one can verify that the quotient $B/\llrr{s_1,s_2,s_3,s_4}$ has
  socle-dimension $2$, for $(n_1,n_2,n_3) = (4,4,2)$ and sufficiently
  general $s_1,s_2,s_3,s_4 \in \Soc B$.
  
  Furthermore, setting $d_B := \dim_{\kk} B$, the examples in Table
  \ref{table} satisfy the inequality $\dim_{\kk} \Hom_B(J,B) < 4
  d_B$. For $(n_1,n_2,n_3) \neq (4,4,2)$, we have the stronger
  inequality
  \[
  \dim_{\kk} \Hom_B(J,B) < 3 (d_B-1).
  \]
  These examples may define singular points $[J]$, however, since the
  obstruction spaces $T^2(B/\kk,B)_{\ge 0}$ are nontrivial. Thus, the
  examples from Table \ref{table} define (possibly singular) points
  which lie exclusively on elementary components of
  $\hilb^{d_B}(\bbA^4)$ with dimensions less than that of the main
  component; moreover, with the exception of $(4,4,2)$, Table
  \ref{table} provides additional examples which answer Question
  \ref{q:very-small-dim}.
\end{remark}

\begin{table}[ht] 
  \caption{Triples defining further ideals with trivial negative
    tangents} \centering
  \begin{tabular}{|c c c c c c c|}
    \hline
    \multicolumn{7}{|c|}{$(n_1,n_2,n_3)$} \\
    \hline
    (4,4,2) & (4,5,3) &&&&& \\
    (5,5,3) & (5,6,3--4) & (5,7,4) &&&& \\
    (6,6,3--4) & (6,7,4--5) & (6,8,4--5) & (6,9,5) &&& \\
    (7,7,4--5) & (7,8,4--6) & (7,9,5--6) & (7,10,5--6) & (7,11,6) && \\
    (8,8,4--6) & (8,9,5--7) & (8,10,5--7) & (8,11,6--7) & (8,12,6--7) & (8,13,7) & \\
    (9,9,5--7) & (9,10,5--8) & (9,11,6--8) & (9,12,6--8) & (9,13,7--8) & (9,14,7--8) & (9,15,8) \\
    \hline
  \end{tabular}
  \label{table}
\end{table}

\begin{remark}
  \label{rmk:nestedHil}
  Theorem \ref{thm:intro2} also allows one to produce natural points
  on the \emph{nested} Hilbert scheme, which parametrizes flags of
  ideals. Specifically, letting $I^{(i)}=I+\llrr{s_1,\dots,s_i}$, we
  see that
  \[
  [I^{(r)}\supset\dotsb\supset I^{(1)}\supset I] \in
  \hilb^{(d-r,\dots,d-1,d)}(\bbA^4).
  \]
  Preliminary investigations suggest the points $[J\supset I] \in
  \hilb^{(d-1,d)}(\bbA^4)$ are smooth.  This raises the question:
  does $[I^{(r)}\supset\dotsb\supset I^{(1)}\supset I]$ always define
  a smooth point
  of $\hilb^{(d-r,\dots,d-1,d)}(\bbA^4)$?
\end{remark}


 \bibliography{elementary}{} \bibliographystyle{amsalpha}

\end{document}